\documentclass[11pt]{article}

\usepackage{amsmath,amssymb,amsthm,amsfonts}
\usepackage{cite,color,enumerate}
\usepackage{geometry}

\newcommand{\eps}{\varepsilon}

\newtheorem{lemma}{Lemma}[section]
\newtheorem{prop}[lemma]{Proposition}
\newtheorem{defn}[lemma]{Definition}
\newtheorem{cor}[lemma]{Corollary}
\newtheorem{thm}{Theorem}
\newtheorem{rmk}[lemma]{Remark}

\numberwithin{equation}{section}

\begin{document}

\title{\bf Threshold phenomena for symmetric-decreasing radial
  solutions of reaction-diffusion equations}

\author{C. B. Muratov\thanks{Department of Mathematical Sciences, New
    Jersey Institute of Technology, Newark, NJ 07102, USA. }  \and
  X. Zhong$^{*,}$\thanks{The author is deceased.}}

\maketitle

\begin{abstract}
  We study the long time behavior of positive solutions of the Cauchy
  problem for nonlinear reaction-diffusion equations in $\mathbb{R}^N$
  with bistable, ignition or monostable nonlinearities that exhibit
  threshold behavior. For $L^2$ initial data that are radial and
  non-increasing as a function of the distance to the origin, we
  characterize the ignition behavior in terms of the long time
  behavior of the energy associated with the solution. We then use
  this characterization to establish existence of a sharp threshold
  for monotone families of initial data in the considered class under
  various assumptions on the nonlinearities and spatial dimension. We
  also prove that for more general initial data that are sufficiently
  localized the solutions that exhibit ignition behavior propagate in
  all directions with the asymptotic speed equal to that of the unique
  one-dimensional variational traveling wave.
\end{abstract}



\section{Introduction}

This paper is a continuation of our earlier work in \cite{mz:nodea13},
in which we considered a one-dimensional version of the Cauchy problem
for the reaction-diffusion equation
\begin{equation}\label{main}
  u_t = \Delta u +f(u), \qquad
  x \in \mathbb{R}^N, \; t>0,
\end{equation}
with initial condition
\begin{equation}\label{initial}
  u(x, 0)=\phi(x) \geq 0,
  \qquad \phi \in L^2(\mathbb{R}^N) \cap L^{\infty}(\mathbb{R}^N).
\end{equation}
Here $u = u(x, t) \in [0, \infty)$, and the nonlinearity $f$ is of
monostable, ignition or bistable type (for a review, see, e.g.,
\cite{xin00}). For all three nonlinearity types, $f$ satisfies
\begin{equation}
 \label{bistable}
 f \in C^1[0, \infty), \qquad f(0)=f(\theta_0)=f(1)=0, \qquad
 f(u)\left\{\!\!\! 
   \begin{array}{ll} 
     \leq 0, & in \; [0,\theta_0] \cup(1,\infty),\\
     >0, & in \; (\theta_0,1), 
   \end{array}\right. 
\end{equation} for
some $\theta_0\in[0,1)$. This type
of problems appears in various applications in physics, chemistry and
biology \cite{merzhanov99,murray,mikhailov,ko:book}. As a prototypical
nonlinearity, one may 
consider
\begin{align}
  \label{fnagumo}
  f(u) = u ( 1 - u ) (u - \theta_0),
\end{align}
which gives rise to what is sometimes called Nagumo's equation
\cite{nagumo,mckean70} and is also a particular version of the
Allen-Cahn equation \cite{allen79}. Moreover, in the case when
$\theta_0 > 0$ we assume that the $u = 1$ equilibrium is more
energetically favorable than the $u = 0$ equilibrium, i.e., that
\begin{equation}
  \label{negpotential}
  \int_{0}^{1}f(s) \, ds > 0.
\end{equation}
For the nonlinearity $f$ from \eqref{fnagumo}, the condition in
\eqref{negpotential} corresponds to $\theta_0 < \frac12$. Note that
such nonlinearities are often known to exhibit {\em ground states},
i.e., positive variational solutions of (for a precise definition used
in our paper, see Definition \ref{d:gs})
\begin{equation}\label{stationary}
  \Delta v +f(v)=0, \qquad x \in \mathbb R^N.
\end{equation} 

For the problem with $N = 1$ and unbalanced bistable nonlinearities,
i.e., the nonlinearities satisfying \eqref{bistable} and
\eqref{negpotential} for which $\theta_0 > 0$ and $f(u) < 0$ for all
$u \in (0, \theta_0)$, we proved, under some mild non-degeneracy
assumption for the nonlinearity $f$ near zero, that there are exactly
three alternatives for the long-time behavior of solutions of
\eqref{main} with symmetric-decreasing initial data satisfying
\eqref{initial} \cite{mz:nodea13}:
\begin{enumerate}[~ --]
\item ignition, when the solution converges locally uniformly to the
  equilibrium $u = 1$; 
\item extinction, when the solution converges uniformly to the
  equilibrium $u = 0$;
\item convergence to the unique ground state $v$ centered at the
  origin.
\end{enumerate}
The solution corresponding to the third alternative serves as a kind
of separatrix between the extinction and the ignition behaviors for
monotone families of initial data and may be referred to as the {\em
  threshold} solution. Moreover, this solution exhibits a {\em sharp}
threshold behavior, in the sense that for any strictly increasing
family of initial data exhibiting extinction for sufficiently small
values of the parameter and ignition for sufficiently large values of
the parameter there is exactly one member of the family that gives
rise to a threshold solution. Similar results were also obtained for
the case of monostable and ignition nonlinearities \cite{mz:nodea13}.

We note that studies of the long time behavior of solutions of
\eqref{main} go back to the classical work of Fife \cite{fife79arma},
in which all possible long-time behaviors of solutions of \eqref{main}
in one space dimension were classified for a general class of initial
data for nonlinearities like the one in \eqref{fnagumo} (for related
studies, see also \cite{F1997,FP1997,fasangova98}). Studies of the
threshold behavior go back to Kanel' \cite{kanel62}, and more recently
to those by Zlato\v{s} \cite{zlatos06}, Du and Matano \cite{du10} and
Pol\'{a}\v{c}ik \cite{P2011}, who established sharpness of the
threshold in a number of general settings. In particular, for $N = 1$
and bistable nonlinearities Du and Matano proved that one of the three
alternatives stated earlier holds for arbitrary bounded, compactly
supported initial data, provided that the ground state $v$ is suitably
translated. Among other things, for $N > 1$ and bistable
nonlinearities with $f'(0) < 0$ Pol\'{a}\v{c}ik showed, still for
compactly supported initial data, the existence of a sharp threshold
and that the threshold solution becomes asymptotically radial and
symmetric-decreasing relative to some point $x^* \in \mathbb R^N$ as
$t \to \infty$. Combining this result with those of \cite{busca02}
(see also \cite{galaktionov07} for a related work), one can further
conclude that the threshold solution converges to a ground state. We
note that in the considered situation the case of non-symmetric
initial data that do not have a sufficiently fast (exponential) decay
at infinity remains open, even in one space dimension.

For $N \geq 2$, the problem of classifying the long time behaviors for
solutions of \eqref{main} with nonlinearities as in \eqref{fnagumo}
was treated by Jones \cite{J1983}. For radial non-increasing initial
data with values in the unit interval and crossing the threshold value
of $u = \theta_0$, Jones used dynamical systems arguments to prove
that the $\omega$-limit set of each solution consists only of the
stable homogeneous equilibria $u = 0$ and $u = 1$, and of ground
states. Under an extra assumption that the set of all ground states is
discrete, Jones' analysis shows that any solution of the initial value
problem considered in \cite{J1983} converges either uniformly to
$u = 0$, or locally uniformly to $u = 1$, or uniformly to one of the
ground states as $t \to \infty$ (however, for existence of
non-convergent solutions in a related context, see
\cite{polacik96}). Alternatively, convergence to a ground state as the
third alternative follows from the results of \cite{busca02} for
exponentially decaying initial data (the latter assumption is dropped
in a recent work \cite{foldes11}). We note that in contrast to the
$N = 1$ case, in higher dimensions it is generally not known whether
or not \eqref{stationary} may exhibit continuous families of ground
states, even for non-degenerate bistable nonlinearities (for examples
of nonlinearities exhibiting arbitrarily large numbers of distinct
ground states, see \cite{bamon00}). Some general sufficient conditions
establishing the absence of multiplicity of the ground states were
provided by Serrin and Tang \cite{serrin00} (existence of such
solutions under very general assumptions on $f$ goes back to the
classical works of Berestycki and Lions \cite{BL1983} and of
Berestycki, Lions and Peletier \cite{berestycki81}; the possibility of
multiple ground states for $N \geq 2$ and nonlinearities having zero
as a locally stable equilibrium was pointed out in
\cite{polacik01}). In particular, the results of Serrin and Tang apply
to the nonlinearity in \eqref{fnagumo}, thus establishing the expected
multiplicity of the long time behaviors for Nagumo's equation for
radial symmetric-decreasing data in all dimensions, with the {\em
  unique} ground state as the limit of the threshold solution. Another
example of a bistable nonlinearity to which the uniqueness result in
\cite{serrin00} applies is
\begin{align}
  \label{eq:fstdeg}
  f(u) = -u^r + (1 + \gamma) u^p - \gamma u^q, \qquad 1 < r < p < q,
  \qquad \gamma > {(p - r) (q + 1) \over (q - p) (r + 1)},
\end{align}
which satisfies \eqref{bistable} and \eqref{negpotential}.

At the same time, for monostable nonlinearities such a conclusion
about the ground state multiplicity is easily seen to be false. For
example, if $N \geq 3$ and $f(u) = u^p$ for $u \leq \frac12$, with
$p = p_S$, where $p_S := (N + 2) / (N - 2)$ is the Sobolev critical
exponent (here and in the rest of the paper, we use the notations of
\cite{QS2007} for the critical exponents; for the reader's
convenience, the values of the relevant critical exponents are
collected in Table \ref{t:exps}), one gets a continuous family of
ground states
\begin{align}
  \label{eq:vEF}
  v_\lambda(x) := \left( \lambda  + {|x|^2 \over \lambda N (N -
  2)} \right)^{-(N-2)/2},
\end{align}
for all $\lambda \in [2^{2/(N-2)}, \infty)$. Here $v_\lambda$ are the
unique, up to translations, ground states such that
$\|v_\lambda \|_{L^\infty(\mathbb R^N)} \leq \frac12$
\cite{caffarelli89}.  Very recently, Pol\'{a}\v{c}ik and Yanagida
showed that the $\omega$-limit sets for such problems may not consist
only of stationary solutions, even in the radial case
\cite{polacik14}. Therefore, the long time behavior of solutions is
expected to be more delicate in the case of monostable nonlinearities.

In this paper, we revisit the problem of threshold behavior for radial
symmetric-decreasing solutions of \eqref{main} in dimensions $N > 1$
whose studies were initiated by Jones for bistable nonlinearities. Our
main contribution in the latter case is to remove the strong
non-degeneracy assumptions of \cite{J1983,P2011}, which read
$f'(0) < 0$ and $f'(1) < 0$ in the context of the nonlinearities
considered in this paper, and to establish the picture of sharp
threshold behavior for radial symmetric-decreasing $L^2$ initial data,
under a generic assumption on the structure of the set of all ground
states. Note that our results are new even in the case $N = 1$, since,
in contrast with \cite{mz:nodea13}, we do not impose any
non-degeneracy assumptions on $f$ any more, at the expense of not
being able to determine precisely the limit energy of the threshold
solution. In addition, to the best of our knowledge this is a first
general study of threshold phenomena for ignition and monostable
nonlinearities for $N > 1$. In particular, we show that the character
of the threshold behavior depends rather delicately on the dimension
of space and may become quite intricate for $N \geq 3$.


Our paper is organized as follows. In Sec. \ref{s:state}, we discuss
the motivations for our results and present the precise statements in
Theorems \ref{t:radfront}--\ref{t:mono31}. In Sec. \ref{s:prelim}, we
present a number of auxiliary results. In Sec. \ref{s:prop}, we prove
Theorems \ref{t:radfront}--\ref{t:omls} that are concerned with
ignition and propagation phenomena. In Sec. \ref{s:bistab}, we prove
Theorem \ref{t:thrbist} that treats bistable nonlinearities.  In
Sec. \ref{s:igni}, we prove Theorems \ref{t:thrign2} and
\ref{t:thrign3} dealing with ignition nonlinearities. Finally, in
Sec. \ref{s:mono} we prove Theorems \ref{t:mono2}--\ref{t:mono31}
treating monostable nonlinearities.

\begin{table}
  \centering
  \caption{List of critical exponents.}
  \vspace{3mm}
  \begin{tabular}{|c|c|c|c|}
    \hline
    Name & Exponent & Validity & Value for $N = 3$ \\
    \hline \hline
    Fujita & $p_F = (N + 2)/N$ & $N \geq 1$ & $5/3$  \\
    Serrin & $p_{sg} = N / (N - 2)$ & $N \geq 3$ & 3 \\
    Sobolev & $p_S = (N + 2) / (N - 2)$ & $N \geq 3$ & 5 \\
    Joseph-Lundgren & $p_{JL} = 1 + 4/ \left( N - 4 - 2 \sqrt{N - 1} \,
                      \right)$ & $N \geq 11$ & -- \\
    \hline
 \end{tabular}
\label{t:exps}
\end{table}

\section{Statement of results}
\label{s:state}

Our approach to the problem takes advantage of two variational
structures possessed by \eqref{main}. The first one is well known and
expresses the fact that \eqref{main} is an $L^2$ gradient flow
generated by the energy (for justification of this and the following
statements, see the next section)
\begin{align}
  \label{eq:E}
  E[u] := \int_{\mathbb R^N} \left( \frac12 |\nabla u|^2 + V(u)
  \right) dx, \qquad V(u) := -\int_0^u f(s) \, ds,
\end{align}
which is well-defined for all
$u \in H^1(\mathbb R^N) \cap L^\infty(\mathbb R^N)$.  In particular,
this implies the energy dissipation identity for solutions $u(x, t)$
of \eqref{main} that belong to
$H^2(\mathbb R^N) \cap L^\infty(\mathbb R^N)$ for each $t > 0$:
\begin{align}
  \label{eq:dEdt}
  {d E[u(\cdot, t)] \over dt} = -\int_{\mathbb R^N} u_t^2(x, t) \, dx,
\end{align}
and, therefore, the energy evaluated on solutions of \eqref{main} is
non-increasing in time. Yet, in contrast to problems on bounded
domains, $E$ does not serve as a Lyapunov functional for \eqref{main},
since it is not bounded from below a priori.

From \eqref{eq:dEdt}, one easily deduces that whenever
$\lim_{t \to \infty} E[u(\cdot, t)] \not= -\infty$, the $\omega$-limit
set of $u(x, t)$ may consist only of stationary solutions of
\eqref{main}. Indeed, in this case there exists a sequence of
$t_n \in [n, n+1)$ such that $u_t(\cdot, t_n) \to 0$ in
$L^2(\mathbb R^N)$ as $n \to \infty$.  Therefore, multiplying
\eqref{main} by a test function $\varphi \in C^\infty_c(\mathbb R^N)$
and integrating, we can see from the obtained equation:
\begin{align}
  \label{eq:mainweak}
  \int_{\mathbb R^N} \varphi u_t \, dx = -\int_{\mathbb R^N} \left(
  \nabla u \cdot \nabla \varphi - f(u) \varphi \right) \, dx,
\end{align}
that if $u(\cdot, t_n)$ converges to some limit in
$H^1_{loc}(\mathbb R^N)$, that limit satisfies \eqref{stationary}
distributionally (hence also classically \cite{gilbarg}). In view of
the standard parabolic regularity, the latter is true, at least on a
subsequence of $t_{n_k} \to \infty$.  Furthermore, if this limit is
independent of the subsequence, then by the uniform in space H\"older
regularity of $u(x, \cdot)$ (see Proposition \ref{holder} below) the
obtained limit is a full limit as $t \to \infty$ locally
uniformly. Nevertheless, despite the energy $E[u(\cdot, t_n)]$ being
bounded from below in this situation for all $n$, we cannot yet
conclude that the obtained limit is a critical point of $E$, in the
sense that the limit has {\em finite} energy. In this paper, we refer
to those solutions of \eqref{stationary} that do as ground
states. More precisely, we have the following definition.

\begin{defn}
  \label{d:gs}
  We call $v \in C^2(\mathbb R^N)$ solving \eqref{stationary} a {\em
    ground state}, if $v > 0$, $v(x) \to 0$ as $|x| \to \infty$,
  $|\nabla v| \in L^2(\mathbb R^N)$ and $V(v) \in L^1(\mathbb R^N)$.
\end{defn}

One naturally expects that for a variety of initial data the solutions
of \eqref{main} go locally uniformly to the equilibrium $u = 1$, whose
energy under \eqref{negpotential} is formally equal to negative
infinity. The latter is intimately related to the phenomenon of {\em
  propagation}, whereby the solution at long times may look
asymptotically like a radially divergent front invading the $u = 0$
equilibrium by the $u = 1$ equilibrium with finite propagation speed,
even for non-radial initial data \cite{jones83,J1983,aronson78}.

To discern between different classes of long time limit behaviors of
solutions of \eqref{main}, it is useful to take advantage of a
different variational structure of \eqref{main} that was pointed out
in \cite{M2004}. In the case of radial solutions of \eqref{main}, we
may formulate this variational structure as follows. Let
$x = (y, z) \in \mathbb R^N$, where $y \in \mathbb R^{N-1}$ and
$z \in \mathbb R$ (this notation is used throughout the rest of the
paper). For a fixed $c > 0$, define
\begin{align}
  \label{eq:utilde}
  \tilde u(y, z, t) := u(y, z + ct, t),
\end{align}
which corresponds to $u(x, t)$ in the reference frame moving with
constant speed $c$ in the $z$-direction.  Then \eqref{main} written in
terms of $\tilde u$ takes the following form:
\begin{align}
  \label{mainc}
  \tilde u_t = \Delta \tilde u + c \tilde u_z + f(\tilde u).
\end{align}
This equation is a gradient flow in the exponentially weighted space
$L^2_c(\mathbb R^N)$, defined to be the completion of
$C^\infty_c(\mathbb R^N)$ with respect to the norm
\begin{align}
  \label{eq:L2cnorm}
  \| u \|_{L^2_c(\mathbb R^N)} := \left( \int_{\mathbb R^N} e^{cz}
  |u|^2 dx \right)^{1/2},
\end{align}
and is generated by the functional 
\begin{equation}\label{weighted}
  \Phi_c[u]:=\int_{\mathbb{R}^{N}}
  e^{cz} \left( \frac{1}{2}|\nabla u|^2+V(u) \right) dx,
\end{equation}
which is well-defined for all
$u \in H^1_c(\mathbb R^N) \cap L^\infty(\mathbb R^N)$, where
$H^1_c(\mathbb R^N)$ is the exponentially weighted Sobolev space
similarly obtained from $C^\infty_c(\mathbb R^N)$ via completion with
respect to the norm
\begin{align}
  \label{eq:H1cnorm}
  \| u \|_{H^1_c(\mathbb R^N)} := \left( \| u \|_{L^2_c(\mathbb
  R^N)}^2 + \| \nabla u \|_{L^2_c(\mathbb R^N)}^2 \right)^{1/2}.
\end{align}
Note that the space obtained in this way is a Hilbert space with the
naturally defined inner product. 

The above formulation captures {\em propagation} of solutions of
\eqref{main} \cite{M2004,mn:cms08}. Notice that in the radial context
we arbitrarily chose the last component of $x$ as the axis of
propagation. More generally, one can still use the above variational
structure to analyse propagation in an arbitrary direction in
$\mathbb R^N$ by rotating the initial condition appropriately. The
dissipation identity for the solutions of \eqref{mainc} that belong to
$H^2_c(\mathbb R^N)$ (the space of all functions in
$H^1_c(\mathbb R^N)$ whose first derivatives also belong to
$H^1_c(\mathbb R^N)$) takes the form:
\begin{align}
  \label{eq:dPhicdt}
  {d \Phi_c[\tilde u(\cdot, t)] \over dt} = - \int_{\mathbb R^N}
  e^{cz} \tilde u_t^2(\cdot, t) \, dx.
\end{align}
The constant $c > 0$ above is arbitrary and can be suitably chosen for
the purposes of the analysis. One particular value of $c$ is special,
however.

\begin{prop}
  \label{p:tw}
  Let $N = 1$ and let \eqref{bistable} hold with some
  $\theta_0 \in [0, 1)$. Also let $f'(0) = 0$ if $\theta_0 = 0$, or
  let \eqref{negpotential} hold if $\theta_0 > 0$. Then there exists a
  unique $c^\dag > 0$ and a unique
  $\bar u \in C^2(\mathbb R) \cap H^1_{c^\dag}(\mathbb R)$ such that
  $0 < \bar u < 1$, $\bar u' < 0$, $\bar u(+\infty) = 0$,
  $\bar u(-\infty) = 1$, $\bar u(0) = \frac12$, and $\bar u$ minimizes
  $\Phi_{c^\dag}$ over all $u \in H^1_{c^\dag}(\mathbb R)$ such that
  $0 \leq u \leq 1$. Furthermore, $u(x, t) = \bar u(x - c^\dag t)$
  solves \eqref{main}.
\end{prop}

\noindent This proposition is an immediate corollary to
\cite[Proposition 2.3]{mz:nodea13}. The solution $u(x, t)$ in
Proposition \ref{p:tw} is an example of a {\em variational traveling
  wave} and plays an important role for the long time behavior of
solutions of \eqref{main} \cite{M2004,mn:cms08}. Its existence allows
us to make a very general conclusion about propagation of the trailing
and the leading edges of the solution with localized initial data. For
$\delta \in (0, 1)$, we define
\begin{align}
  R_\delta^+(t) & := \sup_{x \in \mathbb R^N} \{ |x| \ : \ u(x, t) >
                  \delta \},   \label{eq:Rdp} \\
  R_\delta^-(t) & := \inf_{x \in \mathbb R^N} \{ |x| \ : \ u(x, t) <
                  \delta \}. \label{eq:Rdm} 
\end{align}
The functions $R^\pm_\delta(t)$ represent, respectively, the positions
of the leading and the trailing edges of radially divergent solutions
at level $\delta$. Then we have the following result, which is a
consequence of the gradient flow structure generated by $\Phi_c$.

\begin{thm}[Propagation]
  \label{t:radfront}
  Let \eqref{bistable} hold with some $\theta_0 \in [0, 1)$, and let
  $f'(0) = 0$ if $\theta_0 = 0$, or let \eqref{negpotential} hold if
  $\theta_0 > 0$. Assume that $u(x, t)$ is a solution of \eqref{main}
  satisfying \eqref{initial} with $Q(\phi) \in L^2_c(\mathbb R^N)$ for
  some $c > c^\dag$ and every rotation $Q$, and that
  $u(\cdot, t) \to 1$ locally uniformly as $t \to \infty$. Then
  \begin{align}
    \label{eq:Rdinfty}
    \lim_{t \to \infty} {R^\pm_\delta(t) \over t} = c^\dag.
  \end{align}
\end{thm}
\noindent Here, as usual, the rotation map $Q$ is defined via
$Q(\phi(x)) := \phi(Ax)$ for some $A \in SO(N)$. We note in passing
that the same result is well known for $\theta_0 > 0$, or for
$\theta_0 = 0$ and $f'(0) > 0$, in the case of compactly supported
initial data \cite{aronson78}. In particular, in the latter case the
problem exhibits {\em hair-trigger effect}, i.e., any non-zero initial
data gives rise to the solution that converges locally uniformly to
1. Therefore, assuming $f'(0) \leq 0$ throughout our paper is not
really a restriction.

Let us note that for $c \geq c^\dag$ the functional $\Phi_c[u]$ is
bounded from below by zero for all $u \in H^1_c(\mathbb R^N)$
\cite{mn:cms08}.  Therefore, it would be natural to try to use the
monotone decrease of $\Phi_c$ evaluated on the solution of
\eqref{mainc} to establish convergence of solutions of \eqref{mainc}
to traveling fronts. This is indeed possible in the case $N = 1$,
provided that $f'(0) \leq 0$ and $f'(1) < 0$ in addition to
\eqref{bistable} and \eqref{negpotential}. In this case the solutions
of \eqref{mainc} with front-like initial data converge exponentially
fast to a translate of the one-dimensional non-trivial minimizer of
$\Phi_c$ \cite{mn:sima12}. However, for $N > 1$ it is known that
solutions of \eqref{main} with bistable nonlinearities go to zero
locally uniformly in the reference frame moving with speed $c^\dag$
\cite{uchiyama85,roussier04}.

\begin{rmk}
  \label{r:lb}
  Removing the assumption that $\phi \in L^2_c(\mathbb R^N)$ in
  Theorem \ref{t:radfront}, one still has
  \begin{align}
    \label{eq:Rdinftyp}
    \liminf_{t \to \infty} {R^\pm_\delta(t) \over t} \geq c^\dag.
  \end{align}
\end{rmk}

From Theorem \ref{t:radfront} and Remark \ref{r:lb}, one can see that
under our assumptions on $f$ the ignition behavior implies propagation
for general $L^2$ initial data. We now consider further implications
of propagation for {\em radial symmetric-decreasing} data.
\begin{enumerate}
\item[(SD)] The initial condition $\phi(x)$ in (\ref{initial}) is
  radial symmetric-decreasing, i.e., $\phi(x)=g(|x|)$ for some $g(r)$
  that is non-increasing for every $r>0$.
\end{enumerate}
Note that the slight abuse of notation in the definition (SD) is not a
problem, since the solution $u(x, t)$ of \eqref{main} satisfying
\eqref{initial} and (SD) is a strictly decreasing function of $|x|$
for all $t > 0$. We will show that for initial data obeying (SD),
propagation implies that the energy dissipation rate cannot vanish,
which means that ignition always leads to the energy not being bounded
from below. In fact, the converse also holds. This leads to the
following result which characterizes the ignition scenario via the
asymptotic behavior of the energy evaluated on solutions of
\eqref{main}.

\begin{thm}[Ignition]
  \label{t:igni}
  Let \eqref{bistable} hold with some $\theta_0 \in [0, 1)$, and let
  $f'(0) = 0$ if $\theta_0 = 0$, or let \eqref{negpotential} hold if
  $\theta_0 > 0$. Assume that $u(x, t)$ is a solution of \eqref{main}
  satisfying \eqref{initial} with (SD). Then:
  \begin{enumerate}[(i)]
  \item If $u(\cdot, t_n) \to 1$ locally uniformly in $\mathbb R^N$
    for
    some sequence of $t_n \to \infty$, then \\
    \mbox{$\displaystyle \lim_{t \to \infty} E[u(\cdot, t)] =
      -\infty$.}
  \item If $\displaystyle \lim_{t \to \infty} E[u(\cdot, t)] < 0$,
    then $u(\cdot, t) \to 1$ locally uniformly in $\mathbb R^N$ as
    $t \to \infty$.
  \end{enumerate}
\end{thm}

The main implication of Theorem \ref{t:igni} is that it excludes the
possibility of the equilibrium $u = 1$ to be the long time limit of
solutions of \eqref{main} with energy bounded from below. Hence, for
initial data satisfying (SD) the remaining possibilities are radial
non-increasing solutions of \eqref{stationary}.  If $v$ is such a
solution, it satisfies an ordinary differential equation in $r = |x|$
and can be parametrized by its value at the origin. More precisely, if
$\mu \in [0, 1)$ is such that $\mu = v(0)$, then $v(x) = v_\mu(|x|)$,
where $v_\mu \geq 0$ satisfies for all $0 < r < \infty$
\begin{align}
  \label{eq:vmu}
  v_\mu''(r) + {N - 1 \over r} \, v_\mu'(r) + f(v_\mu(r)) = 0, \qquad
  v_\mu'(r) \leq 0, \qquad
  v_\mu(0) = \mu, \quad v_\mu'(0) = 0. 
\end{align}
It is easy to see that all solutions of \eqref{eq:vmu} are either
identically constant (equal to a zero of $f$), or are strictly
decreasing and approaching a zero of $f$ as $r \to \infty$.

Since ground states in the sense of Definition \ref{d:gs} are a
particular class of solutions of \eqref{stationary} that play a
special role for the long time limits of \eqref{main}, we introduce
the notation
\begin{align}
  \label{eq:yps}
  \Upsilon := \{ \mu \in (0, 1) \ : \ v_\mu(|x|) \text{ is a ground
  state} \}.  
\end{align}
Recall that in many particular situations the set $\Upsilon$ is
generically expected to be a {\em discrete} set of points, possibly
consisting of only a single point, as is the case for the
nonlinearities in \eqref{fnagumo} or \eqref{eq:fstdeg}. Under this
condition, convergence to a ground state becomes full convergence,
rather than sequential convergence, as $t \to \infty$.  This
conclusion will be seen to remain true for bistable and ignition
nonlinearities under the following more general assumption:

\begin{enumerate}
\item[(TD)] The set $\Upsilon$ is totally disconnected. 
\end{enumerate}

\noindent By a totally disconnected set, we understand a set whose
connected components are one-point sets. We note that verifying (TD)
in practice may be rather difficult, in view of the quite delicate
structure of the solution set for \eqref{stationary} in its full
generality. Nevertheless, as was already noted above, this condition
is expected to hold generically and allows us to avoid getting into
the specifics of the existence theory for the elliptic equation
\eqref{stationary} and concentrate instead on the evolution problem
associated with \eqref{main}.

As a consequence of Theorem \ref{t:igni} and the gradient flow
structure of \eqref{main}, we have the following general result about
all possible long-time behaviors of solutions of \eqref{main} with
radial symmetric-decreasing initial data in
$L^2(\mathbb R^N) \cap L^\infty(\mathbb R^N)$.

\begin{thm}[Ignition vs. Failure]
  \label{t:omls}
  Let \eqref{bistable} hold with some $\theta_0 \in [0, 1)$, and let
  $f'(0) = 0$ if $\theta_0 = 0$, or let \eqref{negpotential} hold if
  $\theta_0 > 0$. Assume that $u(x, t)$ is a solution of \eqref{main}
  satisfying \eqref{initial} with (SD). Then there are two
  alternatives:
  \begin{enumerate}
  \item $\displaystyle \lim_{t \to \infty} u(\cdot, t) = 1$ locally
    uniformly in $\mathbb R^N$ and
    $\displaystyle \lim_{t \to \infty} E[u(\cdot, t)] = -\infty$.
  \item
    $\displaystyle \liminf_{t \to \infty}\sup_{x \in B_R(0)} \big|
    u(x, t) - v_\mu(|x|) \big| = 0$
    for every $R > 0$ and every $\mu \in I$, where $I = [a,b]$, with
    some $0 \leq a \leq b < 1$, $v_\mu(|x|)$ satisfies \eqref{eq:vmu}
    for all $\mu \in I$, and
    $\displaystyle \lim_{t \to \infty} E[u(\cdot, t)] \geq 0$.
  \end{enumerate}
\end{thm}

We note that more precise conclusions for the second alternative in
Theorem \ref{t:omls} would need further assumptions on the
nonlinearity of the problem, such as those that would yield (TD), or,
perhaps, analyticity of $f(u)$ \cite{simon83,galaktionov07}. Apart
from the first option, we do not pursue this further in the present
paper.

\begin{rmk}
  \label{r:f}
  It is easy to see that the conclusions of all the above theorems
  remain true, if one assumes that $f \in C^1[0, \infty)$,
  $f(0) = f(1) = 0$, $f'(0) \leq 0$, $f(u) \leq 0$ for all $u \geq 1$,
  and that $u_m = 1$ is the only root of $f(u)$ such that
  $V(u_m) < 0$.
\end{rmk}

We now turn our attention to the study of threshold phenomena. We use
the notations similar to those in \cite{du10}. Let
$X := \{\phi(x):\phi(x)\;\text{satisfies (\ref{initial}) and (SD)}\}$,
and let $\lambda^+ > 0$.  We consider a one-parameter family of
initial conditions $\phi_{\lambda} \in X$ with
$\lambda \in [0, \lambda^+]$, satisfying the following conditions:
\begin{enumerate}
\item[(P1)] The map $\lambda\mapsto\phi_{\lambda} \in X$ is continuous
  from $[0, \lambda^+]$ to $L^2(\mathbb{R}^N)$;
\item[(P2)] If $0<\lambda_1<\lambda_2$, then
  $\phi_{\lambda_1}\leq\phi_{\lambda_2}$ and
  $\phi_{\lambda_1}\not= \phi_{\lambda_2}$ in $L^2(\mathbb{R}^N)$.
\item[(P3)] $\phi_0(x)=0$ and $E[\phi_{\lambda^+}] < 0$.
\end{enumerate}
\noindent We denote by $u_\lambda(x, t)$ the solution of \eqref{main}
with the initial datum $\phi_\lambda$. Clearly, $u_0(x, t) = 0$, and
by Theorem \ref{t:igni} we have $u_{\lambda^+} (\cdot, t) \to 1$
locally uniformly as $t \to \infty$. Therefore, the solutions
corresponding to the endpoints of the interval of
$\lambda \in [0, \lambda^+]$ exhibit qualitatively distinct long time
behaviors. We wish to characterize all possible behaviors for
intermediate values of $\lambda$ and, in particular, to determine the
structure of the threshold set.

To proceed, we consider the cases of bistable, ignition and monostable
nonlinearities separately, as they lead to rather different sets of
conclusions. We start with the bistable nonlinearity, namely, the
nonlinearity $f$ satisfying \eqref{bistable} with $\theta_0 > 0$,
together with \eqref{negpotential} and an extra assumption that
$f(u) < 0$ for all $u \in (0, \theta_0)$.  The key observation is that
for these nonlinearities there exists $\theta^* \in (\theta_0, 1)$
such that
\begin{align}
  \label{eq:thstar}
  \int_0^{\theta^*} f(s) \, ds = 0.
\end{align}
Furthermore, we have $V(u) > 0$ for all $0 < u < \theta^*$ and
$V(u) < 0$ for all $\theta^* < u < \theta^{\diamond}$, for some
$\theta^{\diamond} \in (1, \infty]$.  At the same time, the set of all
zeros of $f$ that lie in $[0, 1)$ consists of only two isolated
values: $u = 0$ and $u = \theta_0$.  Therefore, by Theorem
\ref{t:omls}, if the solution with initial data satisfying (SD) does
not converge locally uniformly to $u = 1$, on sequences of times going
to infinity it either converges to $u = 0$, or to a decaying radial
symmetric-decreasing solution $v$ of \eqref{stationary}. Note that for
bistable nonlinearities and $N \geq 3$, all positive solutions of
\eqref{eq:vmu} converging to zero at infinity are ground states (after
extension to $\mathbb R^N$), since every decaying solution of
\eqref{stationary} is subharmonic for $|x| \gg 1$ and, therefore,
decays no slower than $|x|^{2-N}$. In this case the statement in (TD)
concerns all radial decaying solutions of \eqref{stationary}.


The next theorem further details the above picture and also
establishes the existence of a {\em unique} threshold between ignition
and extinction for monotone families of initial data, under (TD).

\begin{thm}[Threshold for Bistable Nonlinearities]
  \label{t:thrbist}
  Let \eqref{bistable} hold with some $\theta_0 \in (0, 1)$, let
  $f(u) < 0$ for all $u \in (0, \theta_0)$ and let
  \eqref{negpotential} hold. Assume that $u_\lambda(x, t)$ are
  solutions of \eqref{main} with the initial data $\phi_\lambda$
  satisfying (P1)--(P3). Then, under (TD) there exists
  $\lambda_* \in (0, \lambda^+)$ such that:
  \begin{enumerate}
  \item $\displaystyle \lim_{t \to \infty} u_\lambda(\cdot, t) = 1$
    locally uniformly in $\mathbb R^N$ and
    $\displaystyle \lim_{t \to \infty} E[u_\lambda (\cdot, t)] =
    -\infty$ for all $\lambda > \lambda_*$.
  \item $\displaystyle \lim_{t \to \infty} u_\lambda(\cdot, t) = 0$
    uniformly in $\mathbb R^N$ and
    $\displaystyle \lim_{t \to \infty} E[u_\lambda(\cdot, t)] = 0$ for
    all $\lambda < \lambda_*$.
  \item $\displaystyle \lim_{t \to \infty} u_\lambda(\cdot, t) = v_* $
    uniformly in $\mathbb R^N$ and
    $\displaystyle \lim_{t \to \infty} E[u_\lambda(\cdot, t)] \geq
    E_0$,
    where $v_*(x) = v_{\mu_*}(|x|)$ and $v_{\mu_*}$ satisfies
    \eqref{eq:vmu}, $\mu_* \in \Upsilon$, and $E_0 := E[v_*] > 0$, for
    $\lambda = \lambda_*$.
  \end{enumerate}
\end{thm}

\noindent Notice that (TD) is the only assumption on the set of radial
symmetric-decreasing solutions of \eqref{stationary} that has been
made in Theorem \ref{t:thrbist}. The fact that (TD) is sufficient is
due to a strong instability of the ground states, which precludes a
possibility of an {\em ordered} family of ground states (for
non-degenerate bistable nonlinearities, this fact was spelled out in
\cite{busca02}). It is interesting whether (TD) can be relaxed, so
that the sharp threshold result holds even when there is a continuum
of ground states.


For ignition nonlinearities, i.e., those that satisfy \eqref{bistable}
with $\theta_0 > 0$ and having $f(u) = 0$ for all
$u \in [0, \theta_0]$, the situation becomes more complicated. Recall
that in the considered setting and with $N = 1$ the threshold solution
is known to converge to the unstable equilibrium solution
$u = \theta_0$ \cite{zlatos06,du10,mz:nodea13}. This happens because
in the case $N = 1$ the only symmetric-decreasing solutions of
\eqref{eq:vmu} that satisfy $0 < v_\mu < 1$ are constant solutions
$v_\mu(r) = \mu$, for any $\mu \in (0, \theta_0]$, and by the well
known property of the heat equation every solution of \eqref{main}
satisfying \eqref{initial} with $\phi \leq \theta_0$ goes to
zero. Hence the solutions of \eqref{main} cannot converge locally
uniformly to any constant solution $0 < v_\mu < \theta_0$. On the
other hand, for $N \leq 2$ it is easy to see that \eqref{eq:vmu} does
not have any non-constant solutions. Thus, the only alternative to
ignition and extinction in this case is convergence to $u = \theta_0$.
With an extra assumption that $f(u)$ is convex in a neighborhood of
$u = \theta_0$, we are then led to the following result.

\begin{thm}[Threshold for Ignition Nonlinearities: Low Dimensions]
  \label{t:thrign2}
  Let $N \leq 2$, let \eqref{bistable} hold with some
  $\theta_0 \in (0, 1)$, let $f(u) = 0$ for all $u \in [0, \theta_0]$
  and let $f(u)$ be convex in some neighborhood of $u= \theta_0$.
  Assume that $u_\lambda(x, t)$ are solutions of \eqref{main} with the
  initial data $\phi_\lambda$ satisfying (P1)--(P3). Then there exists
  $\lambda_* \in (0, \lambda^+)$ such that:
  \begin{enumerate}
  \item $\displaystyle \lim_{t \to \infty} u_\lambda(\cdot, t) = 1$
    locally uniformly in $\mathbb R^N$ and
    $\displaystyle \lim_{t \to \infty} E[u_\lambda(\cdot, t)] =
    -\infty$ for all $\lambda > \lambda_*$.
  \item $\displaystyle \lim_{t \to \infty} u_\lambda(\cdot, t) = 0$
    uniformly in $\mathbb R^N$ and
    $\displaystyle \lim_{t \to \infty} E[u_\lambda(\cdot, t)] = 0$ for
    all $\lambda < \lambda_*$.
  \item
    $\displaystyle \lim_{t \to \infty} u_\lambda(\cdot, t) = \theta_0$
    locally uniformly in $\mathbb R^N$, and
    $\displaystyle \lim_{t \to \infty} E[u_\lambda(\cdot, t)] \geq 0$,
    if $\lambda = \lambda_*$.
  \end{enumerate}
\end{thm}

On the other hand, for $N \geq 3$ the situation becomes much more
complicated, since in this case many solutions of \eqref{eq:vmu}
exist. In fact, by the results of Berestycki and Lions \cite{BL1983},
for every $v^\infty \in [0, \theta_0)$ there exists a solution of
\eqref{eq:vmu} such that $v_\mu(\infty) = v^\infty$. Also, there may
exist non-constant solutions of \eqref{eq:vmu} with
$v^\infty = \theta_0$, even continuous families of such
solutions. Take, for instance, $f(u) = (u - \theta_0)^{p_S}$ for all
$\theta_0 < u < \theta_1$ for some $\theta_1 \in (\theta_0, 1)$.
Dealing with all these situations would lead us away from the main
subject of the paper, so instead we give a rather general sufficient
condition for our results to hold. Namely, we assume that all
non-constant solutions of \eqref{eq:vmu} that converge to $\theta_0$
at infinity are ground states for the problem with the nonlinearity
shifted by $\theta_0$.

\begin{enumerate}
\item[(V)] If $v$ is a radial symmetric-decreasing solution of
  \eqref{stationary}, then it satisfies all the properties of a ground
  state, except $v(x) \to \theta_0$ as $|x| \to \infty$.
\end{enumerate}

\noindent Under this assumption, we are able to exclude all solutions
of \eqref{eq:vmu} with $v^\infty > 0$ as potential long time limits of
\eqref{main}.

\begin{thm}[Threshold for Ignition Nonlinearities: High Dimensions]
  \label{t:thrign3}
  Let $N \geq 3$ and let \eqref{bistable} hold with some
  $\theta_0 \in (0, 1)$, let $f(u) = 0$ for all $u \in [0, \theta_0]$
  and assume (V).  Assume also that $u_\lambda(x, t)$ are solutions of
  \eqref{main} with the initial data $\phi_\lambda$ satisfying
  (P1)--(P3).  Then, under (TD) there exists
  $\lambda_* \in (0, \lambda^+)$ such that:
  \begin{enumerate}
  \item $\displaystyle \lim_{t \to \infty} u_\lambda(\cdot, t) = 1$
    locally uniformly in $\mathbb R^N$ and
    $\displaystyle \lim_{t \to \infty} E[u_\lambda(\cdot, t)] =
    -\infty$ for all $\lambda > \lambda_*$.
  \item $\displaystyle \lim_{t \to \infty} u_\lambda(\cdot, t) = 0$
    uniformly in $\mathbb R^N$ and
    $\displaystyle \lim_{t \to \infty} E[u_\lambda(\cdot, t)] = 0$ for
    all $\lambda < \lambda_*$.
  \item $\displaystyle \lim_{t \to \infty} u_\lambda(\cdot, t) = v_* $
    uniformly in $\mathbb R^N$ and
    $\displaystyle \lim_{t \to \infty} E[u_\lambda(\cdot, t)] \geq
    E_0$,
    where $v_*(x) = v_{\mu_*}(|x|)$ and $v_{\mu_*}$ satisfies
    \eqref{eq:vmu}, $\mu_* \in \Upsilon$, and $E_0 := E[v_*] > 0$, for
    $\lambda = \lambda_*$.
  \end{enumerate}
\end{thm}

\noindent One can see that this situation is more reminiscent of the
bistable case, with ground states taking over the role as the limits
of the threshold solutions. In particular, uniqueness of the ground
state would imply that it attracts the threshold solution uniformly as
$t \to \infty$. Note that no assumption on convexity of the
nonlinearity near $u = \theta_0$ is needed in this case.


Finally, we turn to monostable nonlinearities, i.e., when $f$
satisfies \eqref{bistable} with $\theta_0 = 0$. Here, once again, one
needs to distinguish the cases $N \leq 2$ and $N \geq 3$. Just as in
the case $N = 1$ \cite{mz:nodea13}, there are no solutions of
\eqref{eq:vmu} when $f(u) > 0$ for any $\mu \in (0,1)$ and $N = 2$.
Hence, the threshold behavior becomes particularly simple.

\begin{thm}[Threshold for Monostable Nonlinearities: Low Dimensions]
  \label{t:mono2}
  Let $N \leq 2$ and let \eqref{bistable} hold with $\theta_0 = 0$.
  Assume that $u_\lambda(x, t)$ are solutions of \eqref{main} with the
  initial data $\phi_\lambda$ satisfying (P1)--(P3). Then there exists
  $\lambda_* \in [0, \lambda^+)$ such that:
  \begin{enumerate}
  \item $\displaystyle \lim_{t \to \infty} u_\lambda(\cdot, t) = 1$
    locally uniformly in $\mathbb R^N$ for all $\lambda > \lambda_*$.
  \item $\displaystyle \lim_{t \to \infty} u_\lambda(\cdot, t) = 0$
    uniformly in $\mathbb R^N$ for all $\lambda \leq \lambda_*$.
  \end{enumerate}
\end{thm}

\noindent The possibility of $\lambda_* = 0$ corresponds to the
hair-trigger effect and is always realized when $f'(0) > 0$
\cite{aronson78}. Similarly, hair-trigger effect is still observed in
the case $f'(0) = 0$ when $f(u) \geq c u^{p_F}$ for some $c > 0$ and
all sufficiently small $u$ (see \cite{gui92} and references
therein). At the same time, as was already mentioned in the
introduction, the statement of the theorem becomes non-trivial when
$f(u) = o \left( u^{p_F} \right)$ for $u \ll 1$, in the sense that
$\lambda_* > 0$ for some choices of families of initial data. Note
that here and in the remaining theorems we did not pursue the limit
behavior of the energy, since its analysis for monostable
nonlinearities becomes rather tricky and, at the same time, its
consequences may not be very informative.

The situation becomes considerably more delicate for $N \geq 3$ ,
since in this case many radial, symmetric-decreasing and decaying
solutions of \eqref{stationary} can exist, and their existence and
properties depend quite sensitively on the behavior of $f(u)$ near
zero and the dimension $N$ (for an extensive discussion in the case of
pure power nonlinearities, see \cite{QS2007}). Here our ability to
characterize sharp threshold behaviors relies on the assumption that
all the decaying solutions of \eqref{stationary} be ground states (in
the sense of Definition \ref{d:gs}). The fact that
$|\nabla v| \in L^2(\mathbb R^N)$ for a ground state $v$ gives rise to
a strong instability of $v$, which we also used to establish sharp
threshold results for bistable and ignition nonlinearities. At the
same time, it is known that in the case $N \geq 11$ and pure power
nonlinearities $f(u) = u^p$ with $p \geq p_{JL}$, where
$p_{JL} := 1 + 4/ \left( N - 4 - 2 \sqrt{N - 1} \right)$ is the
Joseph-Lundgren critical exponent, the radial, symmetric-decreasing
and decaying solutions become {\em stable} in a certain sense and form
a monotone increasing continuous family \cite{gui92}. This family of
solutions of \eqref{stationary} clearly produces a counterexample for
the expected sharp threshold behavior for monotone families of data
that do not lie in $L^2(\mathbb R^N)$.

We give two results in which sharp threshold behavior is established
for monostable nonlinearities for $N \geq 3$. We begin with the first
case, in which we assume that there are no solutions of \eqref{eq:vmu}
with $\mu \in (0, 1)$. This situation takes place, for example, when
$f(u) \geq c u^p$ for some $c > 0$ and $p \leq p_{sg}$ for all
$u \ll 1$, where $p_{sg} := N/(N-2)$ is the Serrin critical
exponent. In this situation, \eqref{stationary} is known to have no
positive solutions below $u = 1$ \cite{dancer02,QS2007}.

\begin{thm}[Threshold for Monostable Nonlinearities: High Dimensions,
  Simple] \hfill
  \label{t:mono30}
  Let $N \geq 3$, let \eqref{bistable} hold with $\theta_0 = 0$ and
  assume that \eqref{eq:vmu} has no solutions with $\mu \in (0, 1)$.
  Then, if $u_\lambda(x, t)$ are solutions of \eqref{main} with the
  initial data $\phi_\lambda$ satisfying (P1)--(P3), the conclusion of
  Theorem \ref{t:mono2} holds true.
\end{thm}

\noindent Once again, the result of the theorem is non-trivial, for
example, if $f(u) \simeq c u^p$ with some $c > 0$ and
$p_F < p \leq p_{sg}$ for all $u \ll 1$.

\begin{rmk}
  \label{r:mono30}
  The assumptions of Theorem \ref{t:mono30} also hold, for example,
  for $f(u) = u^p - u^q$ with any $p_{sg} < p \leq p_S$ and $q > p$
  \cite[Theorem 3]{bianchi97}. 
\end{rmk}

\noindent On the other hand, by \cite[Theorem 2]{tang00} the set
$\Upsilon$ for the nonlinearity in Remark \ref{r:mono30} consists of a
single point for all $p_S < p < q$, despite the existence of a
continuous family of positive radial symmetric-decreasing decaying
solutions of \eqref{stationary}. We suspect that in this case, apart
from zero, the unique ground state may still be the only attractor of
the threshold solutions.

We now proceed to the second case. As we already noted, the situation
is quite complex to make detailed conclusions about the sharp
threshold behavior without any further assumptions on $f$ and $N$ in
the monostable case. Here, as in the case of ignition nonlinearities
for $N \geq 3$ we chose to give a rather general sufficient condition
in terms of the properties of positive decaying solutions of
\eqref{eq:vmu}, namely, that they consist only of ground states
(however, for an example of nonlinearities for which this is false,
see \cite{lin88}).  This assumption may be verified with the knowledge
of the asymptotic decay of solutions of \eqref{eq:vmu} at
infinity. Note that existence of ground states for \eqref{stationary}
is known in the case when $f(u) = o\left( u^{p_{S}} \right)$ for
$N \geq 3$ \cite{BL1983}.

\begin{thm}[Threshold Monostable Nonlinearities: High Dimensions,
  Complex] \ \hfill
  \label{t:mono31}
  Let $N \geq 3$, let \eqref{bistable} hold with $\theta_0 = 0$ and
  suppose that every non-constant radial symmetric-decreasing solution
  of \eqref{stationary} is a ground state in the sense of Definition
  \ref{d:gs}.  Assume that $u_\lambda(x, t)$ are solutions of
  \eqref{main} with the initial data $\phi_\lambda$ satisfying
  (P1)--(P3). Then, under (TD) there exists
  $\lambda_* \in [0, \lambda^+)$ such that:
  \begin{enumerate}
  \item $\displaystyle \lim_{t \to \infty} u_\lambda(\cdot, t) = 1$
    locally uniformly in $\mathbb R^N$ for all $\lambda > \lambda_*$.
  \item If $\lambda_* > 0$, then
    $\displaystyle \lim_{t \to \infty} u_\lambda(\cdot, t) = 0$
    uniformly in $\mathbb R^N$ for all $\lambda < \lambda_*$.
  \item For $\lambda = \lambda_*$, there are two alternatives: 
    \begin{enumerate}
    \item $\displaystyle \lim_{t \to \infty} u_\lambda(\cdot, t) = 0$
      uniformly in $\mathbb R^N$.
    \item
      $\displaystyle \lim_{t \to \infty} u_\lambda(\cdot, t) = v_* $
      uniformly in $\mathbb R^N$, where $v_*(x) = v_{\mu_*}(|x|)$ and
      $v_\mu$ satisfies \eqref{eq:vmu} with $\mu_* \in \Upsilon$.
    \end{enumerate}
  \end{enumerate}
\end{thm}

\noindent Assuming that $\lambda_* > 0$, i.e., that the hair-trigger
effect does not occur, the main point of the above theorem is that
under (TD) the threshold is sharp. Yet, we note that one can imagine
rather complex behaviors of the threshold solutions as $t \to \infty$,
if (TD) does not hold. For example, taking $f(u) = u^{p_S}$ for all
$u \leq \frac12$ and $f(u)/u^{p_S}$ decreasing to 0 on $[\frac12, 1]$,
it follows from \cite[Theorem 3]{bianchi97} that all solutions of
\eqref{eq:vmu} with $\mu \in (0, 1)$ are ground states and are given
by \eqref{eq:vEF} with $\lambda \in [2^{2/(N-2)}, \infty)$.  Hence,
our Theorem \ref{t:mono31} does not apply, while Theorem \ref{t:omls}
does.  Here it is not a priori clear whether one could rule out a
threshold solution which approaches the family in \eqref{eq:vEF} with
$\lambda = 1/g(t)$ for some function
$g: [0, \infty) \to (0, 2^{-2/(N-2)})$ that slowly oscillates (with
increasing period) between the two endpoints of its range, approaching
zero on a sequence of times going to infinity. Note, however, that the
more exotic behaviors discussed in \cite{polacik14} cannot occur in
our setting, since we consider $L^2$ initial data.

\section{Preliminaries}
\label{s:prelim}

We start with a basic existence result for \eqref{main} with initial
data from \eqref{initial}. Since we want to take advantage of the
variational structure associated with $\Phi_c$ in \eqref{weighted}, we
also provide an existence result for initial data that lie in the
exponentially weighted spaces. Throughout the rest of the paper,
\eqref{bistable} is always assumed to hold, unless stated otherwise.

\begin{prop}
  \label{p:exist}
  There exists a unique
  $u \in{C^2_1(\mathbb{R}^N\times(0,\infty))} \cap
  L^\infty(\mathbb{R}^N\times(0,\infty))$
  satisfying (\ref{main}) and (\ref{initial}) (using the notations
  from \cite{E1998}), with
  \begin{align}
    \label{eq:L2H2}
    u \in  C([0,\infty);
    L^2(\mathbb{R})) \cap C((0,\infty);H^2(\mathbb{R}^N)),
  \end{align}
  and $u_t \in C((0,\infty);H^1(\mathbb{R}^N))$. Furthermore, if there
  exists $c>0$ such that the initial condition
  $\phi\in{L^2_c(\mathbb{R}^N)} \cap{L^{\infty}(\mathbb{R}^N)}$, then
  the solution of (\ref{main}) and (\ref{initial}) satisfies
  \begin{align}
    \label{eq:L2cH2c}
    u \in C([0,\infty);
    L_c^2(\mathbb{R})) \cap C((0,\infty);H_c^2(\mathbb{R}^N)),
  \end{align}
  with $u_t \in C((0,\infty);H_c^1(\mathbb{R}^N))$.  In addition,
  small variations of the initial data in $L^2(\mathbb{R}^N)$ result
  in small changes of solution in $H^1 (\mathbb{R}^N)$ at any $t > 0$,
  and if $\phi$ satisfies (SD), then so does $u(\cdot, t)$ for all
  $t > 0$.
\end{prop}

\begin{proof}
  The proof follows as in \cite[Chapter 7]{lunardi} and
  \cite[Proposition 3.1]{mn:sima12}, noting that the function
  $\bar u(x, t) = \max \{1, \| \phi \|_{L^\infty(\mathbb R^N)}\}$ is a
  global supersolution. The symmetric decreasing property of $u$
  follows, e.g., from \cite[Proposition 52.17]{QS2007}.
\end{proof}

\noindent Note that the regularity of solutions in Proposition
\ref{p:exist} guaranties that $E[u(\cdot, t)]$ (resp.
$\Phi_c[\tilde u(\cdot, t)]$) is finite, continuously differentiable
and satisfies \eqref{eq:dEdt} on solutions of \eqref{main}
(resp. \eqref{eq:dPhicdt} on solutions of \eqref{mainc}), for any
$t > 0$.

We next recall the classical regularity properties of bounded
solutions of \eqref{main}. Let $D_1 = Q_1 \times [t_1,t_2]$ be an
$(N+1)$-dimensional cylindrical domain in $(x,t)$, where
$Q_1 \subset \mathbb{R}^N$ is open and $t_1 \geq 0$. Let
$Q_2 \subset Q_1$ be open, and assume that there exists
$\varepsilon>0$ such that
\begin{equation}
\bigcup_{x \in Q_2} B_{\varepsilon}(x) \subset Q_1.
\end{equation}
Moreover, let $D_2 = Q_2 \times [t_1+\varepsilon,t_2]$. Then, if
$u(x,t)$ is a solution of \eqref{main}, by Schauder estimates (see,
e.g., \cite{Fr1964}), there exists $C>0$, depending on $\varepsilon$
but independent of $u$ and $D_1$, such that
\begin{equation}
  \label{parreg}
  \|u_t\|_{L^{\infty}(D_2)} + \sum_{1\leq i \leq N} \| \partial_{x_i}
  u \|_{L^{\infty}(D_2)} + \sum_{1\leq i,j \leq
    N} \| \partial_{x_i}\partial_{x_j} u \|_{L^{\infty}(D_2)} \leq C
  \|u\|_{L^{\infty}(D_1)}. 
\end{equation}
We will refer to this boundedness as ``standard parabolic
regularity''. We note that the estimate in \eqref{parreg} also holds
for all classical solutions of \eqref{stationary}, since they can be
trivially considered as time-independent solutions of \eqref{main}.

\begin{cor}
  \label{c:limsup}
  Let $u(x, t)$ be a solution of \eqref{main} satisfying
  \eqref{initial}. Then
  \begin{align}
    \label{eq:limsupu}
    \limsup_{t \to \infty} \| u(\cdot, t) \|_{L^\infty(\mathbb R^N)}
    \leq 1.    
  \end{align}
\end{cor}

\begin{proof}
  By Proposition \ref{p:exist} and standard parabolic regularity, for
  every $T > 0$ the solution $u(x, T)$ is bounded and converges
  uniformly to zero as $|x| \to \infty$. Therefore, if $a > 0$,
  $b > 0$ and
  \begin{align}
    \label{eq:gauss1}
    \bar u(x, t) := 1 + {a \over [4 \pi (t - T + b)]^{N/2}} \exp \left\{
    -{|x|^2 \over 4 (t - T + b)} \right\},
  \end{align}
  then $\bar u(x, t)$ is a supersolution for \eqref{main} for all
  $t \geq T$, and it is possible to choose $a$ and $b$ in such a way
  that $u(x, T) \leq \bar u(x, T)$ for all $x \in \mathbb R^N$. The
  result then follows by comparison principle.
\end{proof}

We now turn to a useful property of solutions of \eqref{main} whose
energy remains bounded for all time. Because of the gradient flow
structure of \eqref{main}, one should expect that such solutions
exhibit ``slowing down'' while approaching stationary solutions on
sequences of times going to infinity. More is true, however, namely,
that a solution with bounded energy also remains close to the limit
stationary solution on a sequence of growing temporal intervals. This
conclusion is a consequence of uniform H\"older continuity of
$u(x, t)$ in $t$ for each $x \in \mathbb R^N$ whenever
$\lim_{t \to \infty} E[u(\cdot, t)] \not= -\infty$ that we establish
below. The result is a basic generalization of the one in
\cite[Proposition 2.8]{mz:nodea13} obtained for $N = 1$.

\begin{prop}
  \label{holder}
  Let $u(x, t)$ be a solution of \eqref{main} satisfying
  \eqref{initial}. If $E[u(\cdot,t)]$ is bounded from below, then
  $u(x,\cdot) \in C^\alpha[T, \infty)$ with
  $\alpha = \frac{1}{2(N+1)}$, for each $x \in \mathbb{R}^N$ and
  $T > 0$. Moreover, the corresponding H\"{o}lder constant of
  $u(x, \cdot)$ converges to $0$ as $T \rightarrow\infty$ uniformly in
  $x$.
\end{prop}
\begin{proof}
  By Proposition \ref{p:exist}, we have $E[u(\cdot, t)] < +\infty$ for
  all $t > 0$, and by \eqref{eq:dEdt} we have that $E[u(\cdot, t)]$ is
  a non-increasing function of $t$. We define
  $E_{\infty} := \lim_{t\rightarrow\infty}E[u(\cdot,t)]$ and observe
  that by our assumptions $E_\infty \in \mathbb R$.  Then, for any
  $x_0 \in \mathbb{R}^N$ and $t_2>t_1\geq T$ we have
  \begin{eqnarray}
    \int_{B_1(x_0)}|u(x,t_2)-u(x,t_1)| \, dx
    &\leq& \int_{t_1}^{t_2}\int_{B_1(x_0)}|u_t(x,t)| \, dx \,
           dt\nonumber\\ 
    &\leq& C_N \sqrt{t_2-t_1} \left( \int_{t_1}^{t_2}\int_{B_1(x_0)}
           u^2_t(x,t) \, dx \, dt \right)^{1/2}\nonumber\\
    &\leq& C_N \sqrt{t_2-t_1} \left( \int_{T}^{\infty}\int_{\mathbb{R}^N}
           u^2_t(x,t) \, dx \, dt \right)^{1/2}\nonumber\\
    &=& C_N \sqrt{(t_2-t_1)  (E[u(\cdot,T)]-E_{\infty})}. 
        \label{sqrtt1t2}
  \end{eqnarray}
  Here and in the rest of the proof $C_N > 0$ is a constant depending
  only on $N$ that changes from line to line.

  On the other hand, by standard parabolic regularity there exists
  $M > 0$ such that
  \begin{equation}
    \|\nabla u(\cdot,t) \|_{L^{\infty}(\mathbb{R}^N)} \leq M, \quad
    \| u(\cdot,t) \|_{L^{\infty}(\mathbb{R}^N)} \leq M  \qquad
    \forall t\geq T.
  \end{equation}
  Without loss of generality we can further assume that
  $u(x_0,t_2)-u(x_0,t_1)\in (0, M]$.  Then, for every
  $x \in B_r(x_0)$, where
  \begin{equation}
    r := \frac{u(x_0,t_2)-u(x_0,t_1)}{2M} \in (0,1),
  \end{equation}
  we have
  \begin{equation}
    u(x,t_2) \geq u(x_0,t_2)-M |x-x_0|
    \geq u(x_0,t_1)+M |x-x_0| \geq u(x,t_1).
  \end{equation}
  This implies that
   \begin{eqnarray}
     \int_{B_1(x_0)}|u(x,t_2)-u(x,t_1)|dx 
     & \geq & \int_{B_r(x_0)}
              (u(x,t_2)-u(x,t_1)) dx
              \nonumber \\ 
     &\geq& \int_{B_r(x_0)}(u(x_0,t_2)-u(x_0,t_1)
            -2M|x-x_0|)dx \nonumber \\
     & = & |u(x_0,t_2)-u(x_0,t_1)| \int_{B_r(x_0)} \left( 1 
           - r^{-1} |x-x_0| \right) dx \nonumber \\
     &=& C_N M^{-N} |u(x_0,t_2)-u(x_0,t_1)|^{N+1}.
   \end{eqnarray}
   Combining this with \eqref{sqrtt1t2} yields
  \begin{equation}
    \label{holdt2t1}
    |u(x_0,t_2)-u(x_0,t_1)|
    \leq C_N \Big( M^{2N} (E[u(\cdot,T)] - E_{\infty})
    \Big)^{\frac{1}{2(N+1)}} (t_2-t_1)^{\frac{1}{2(N+1)}},
  \end{equation}
  i.e., we have $u(x,\cdot) \in C^\alpha[T, \infty)$ for
  $\alpha = {1 \over 2(N+1)}$. Moreover, the limit of the H\"{o}lder
  constant in \eqref{holdt2t1} is
  \begin{equation}
    \lim_{T\rightarrow\infty}
    C_N \left(M^{2N} (E[u(\cdot,T)] -
      E_{\infty})\right)^{\frac{1}{2(N+1)}} =0, 
  \end{equation}
  which completes the proof.
\end{proof}

We will need a technical result about the ground states for
\eqref{stationary}, namely, that all these solutions exhibit a strong
instability with respect to the dynamics governed by \eqref{main}. In
the case $f'(0) < 0$ such a result for all positive solutions of
\eqref{stationary} that decay at infinity is well known (see, e.g.,
\cite[Theorem IV.I]{berestycki81} or \cite[Theorem
5.4]{bates02}). Here we provide a generalization for nonlinearities
that might exhibit a degeneracy near $u = 0$ (for closely related
results, see \cite{cabre04,capella_phd,shi06}).  The key assumption
for the lemma below to hold is that the solution of \eqref{stationary}
has square-integrable first derivatives.

\begin{lemma}
  \label{l:unst}
  Let $f \in C^1[0, \infty)$ and let $v$ be a ground state in the
  sense of Definition \ref{d:gs}. Then there exists $\phi_0^R \geq 0$
  with $\text{supp} \, (\phi_0^R) = B_R(0)$ for some $R > 0$ such that
  $\overline v^\eps(x, t) := v(x) - \eps \phi_0^R(x)$ is a
  supersolution, and
  $\underline v^\eps(x, t) = v(x) + \eps \phi_0^R(x)$ is a
  subsolution, respectively, for \eqref{main}, for all $\eps > 0$
  sufficiently small. 
\end{lemma}

\begin{proof}
  Consider the Schr\"odinger-type operator:
  \begin{equation}
    \mathfrak{L}=- \Delta + \mathcal{V}(x),
    \qquad \mathcal{V}(x) :=-f'(v(x)),
  \end{equation}
  and the associated Rayleigh quotient (for technical background, see,
  e.g., \cite[Chapter 11]{lieb-loss}):
  \begin{equation}
    \label{eq:rayl}
    \mathfrak{R}(\phi) :=\frac{\int_{\mathbb{R}^N} \left( |\nabla \phi|^2 +
        \mathcal{V}(x)   \phi^2 \right) dx}{\int_{\mathbb{R}^N} \phi^2 dx}.
  \end{equation}
  To study the minimization problem for $\mathfrak{R}$, we also
  consider
  \begin{equation}
    \tilde{\mathfrak{L}}=-\Delta + \tilde{\mathcal{V}}(x),
    \qquad \tilde{\mathcal{V}}(x) :=\mathcal{V}(x)+f'(0)=-(f'(v(x))-f'(0)),
  \end{equation}
  with the associated Rayleigh quotient
  \begin{equation}
    \label{eq:raylt}
    \tilde{\mathfrak{R}}(\phi) = \mathfrak{R}(\phi) + f'(0).
  \end{equation}
  Since $\tilde{\mathcal{V}}(x) \in L^\infty(\mathbb R)$ and vanishes
  at infinity, by \cite[Theorem 11.5]{lieb-loss} there exists a
  function $\phi_0 \in H^1(\mathbb{R}^N)$ such that $\phi_0 \neq 0$
  and $\phi_0$ minimizes $\tilde{\mathfrak{R}}$, provided
  \begin{equation}
    \mathfrak{E}_0 := \inf \{\tilde{\mathfrak{R}}(\phi): \;
    \phi \in H^1(\mathbb{R}^N), \;\;
    \phi \neq 0 \} < 0.
  \end{equation}
  Moreover, by \cite[Theorem 11.8]{lieb-loss}, if there exists a
  minimizer $\phi_0 \in H^1(\mathbb{R}^N)$, $\phi_0 \neq 0$, then
  $\phi_0$ can be chosen to be a strictly positive function, and
  $\phi_0$ is unique up to a constant factor. 

  Now, differentiating \eqref{stationary} with respect to $x_i$,
  $i = 1, \ldots, N$, by boundedness of $f'$ on the range of $v$ the
  function $v_i := \partial v / \partial x_i$ satisfies
  \begin{align}
    \label{eq:vi}
    \Delta v_i + f'(v) v_i = 0 \qquad \text{in} \ \mathcal
    D'(\mathbb R^N).
  \end{align}
  Hence, by elliptic regularity we also have
  $v_i \in H^2(\mathbb R^N)$ \cite{gilbarg}. Thus, each $v_i$ is an
  admissible test function in \eqref{eq:rayl}, and by \eqref{eq:vi} we
  have
  \begin{equation}
    \tilde{\mathfrak{R}}(v_i)= f'(0).
  \end{equation}

  Existence of a ground state $v$ implies that $f'(0) \leq 0$
  (otherwise there is hair-trigger effect \cite{aronson78}). Moreover,
  since $v_i$ changes sign, we know that $v_i$ is not a minimizer of
  $\tilde{\mathfrak{R}}$, so
  $\mathfrak{E}_0 < \tilde{\mathfrak{R}}(v_i) \leq 0$, and there
  exists a positive function $\phi_0 \in H^1(\mathbb{R}^N)$ that
  minimizes $\tilde{\mathfrak{R}}$, with
  \begin{equation}
    \min \{\tilde{\mathfrak{R}}(\phi): \;
    \phi \in H^1(\mathbb{R}^N), \;\;
    \phi \neq 0 \} = \tilde{\mathfrak{R}}(\phi_0) < f'(0).
  \end{equation}
  Note that $\phi_0$ also minimizes $\mathfrak{R}$, with
  \begin{equation}
    \min \{\mathfrak{R}(\phi): \;
    \phi \in H^1(\mathbb{R}^N), \;\;
    \phi \neq 0 \} = \mathfrak{R}(\phi_0) =: \nu_0 < 0.
  \end{equation}

  Approximating $\phi_0$ by a function with compact support and using
  it as a test function, we can then see that
  \begin{equation}
    \label{minHR}
    \min \{\mathfrak{R}(\phi): \;
    \phi \in H^1(\mathbb R^N), \ \text{supp} \, (\phi) \subseteq B_R(0), \;\;
    \phi \neq 0 \}=: \nu^R_0<0
  \end{equation} 
  as well for a sufficiently large $R > 0$. In this case, there exists
  a minimizer $\phi^R_0$ of the problem in \eqref{minHR} whose
  restriction to $B_R(0)$ is positive and belongs to
  $H^1_0(B_R(0)) \cap C^2(\bar{B}_R(0))$. Furthermore, we have
  \begin{equation}
    \label{phiR0}
    \mathfrak{L}(\phi^R_0)=\nu^R_0\phi^R_0 \qquad \text{in} \ B_R(0).
  \end{equation}

  Finally, for $\eps > 0$ define
  $\underline{w}(x, t) := \eps \phi_0^R(x)$. Then, using the fact that
  $f(v + \underline{w}) - f(v) = f'(\tilde v) \underline{w}$ for some
  $v \leq \tilde v \leq v + \underline{w}$, we have for all
  $x \in B_R(0)$ and all $\eps > 0$ sufficiently small
  \begin{eqnarray}
    \label{wunderR}
    \underline{w}_t-\Delta \underline{w}-f'(\tilde{v})\underline{w}
    &=&-\Delta \underline{w}-f'(v)\underline{w}
        +(f'(v)-f'(\tilde{v}))\underline{w}\nonumber\\
    &=&\nu^R_0\underline{w}+(f'(v)-f'(\tilde{v}))
        \underline{w}\nonumber\\
    &\leq&\frac{\nu^R_0}{2}\underline{w}\nonumber\\
    &\leq&0.
  \end{eqnarray}
  It is then easy to see that $\underline v^\eps = v + \underline w$
  is a subsolution for \eqref{main}, since $v$ is a solution of
  \eqref{stationary}. 

  The case of $\overline v^\eps$ is treated analogously.
\end{proof}

We note that as a corollary to this result, we have that for
nonlinearities of one sign near the origin there are no {\em ordered}
collections of ground states. Once again, this result is well known in
the case when $f'(0) < 0$ (see, e.g., \cite[Lemma 3.2]{busca02}). More
generally, we have the following statement (for a related result, see
\cite[Theorem 6.1.4]{capella_phd}).

\begin{cor}
  \label{c:nomonot}
  Let $f \in C^1[0, \infty)$, and assume that there exists
  $\alpha > 0$ such that $f$ does not change sign on $(0,
  \alpha)$.
  Let $v_1$ and $v_2$ be two ground states in the sense of Definition
  \ref{d:gs} such that $v_1 \leq v_2$. Then $v_1 = v_2$.
\end{cor}

\begin{proof}
  By strong maximum principle, either $v_1 = v_2$ or $v_1 < v_2$ in
  all of $\mathbb R^N$. We argue by contradiction and assume the
  latter. Let $\underline v_1^\eps$ be the corresponding subsolution
  from Lemma \ref{l:unst} obtained from $v_1$, and choose $\eps > 0$
  so small that $\underline{v}_1^\eps < v_2$. Denote by
  $\underline u_1^\eps$ the classical solution of \eqref{main} with
  $\underline{v}_1^\eps$ as initial datum. Then by comparison principle
  we have $\underline u_1^\eps(x, t) < v_2(x)$ for all
  $x \in \mathbb R^N$ and $t \geq 0$. Existence of such a solution is
  standard. Furthermore, we claim that
  $\underline u_1^\eps(\cdot, t) - v_1 \in H^1(\mathbb R^N)$ for each
  $t \geq 0$, and \eqref{eq:dEdt} holds for $\underline u_1^\eps$.
  Indeed, let $w^\eps := \underline u_1^\eps - v_1$. Then
  $w^\eps(\cdot, 0) \in H^1(\mathbb R^N) \cap L^\infty(\mathbb R^N)$,
  and $w^\eps(x, t)$ solves
  \begin{align}
    \label{eq:wvxt}
    w^\eps_t = \Delta w^\eps + f(v_1 + w^\eps) - f(v_1).
  \end{align}
  Therefore, $w^\eps(x, t)$ satisfies the first half of the
  conclusions of Proposition \ref{p:exist} (cf. \cite[Chapter
  7]{lunardi}). In particular, \eqref{eq:wvxt} is an $L^2$ gradient
  flow generated by the energy
  \begin{align}
    \label{eq:tE}
    \widetilde E[w] := \int_{\mathbb R^N} \left( \frac12 |\nabla w|^2
    + V(v_1 + w) - V(v_1) - V'(v_1) w \right) dx,
  \end{align}
  easily seen to be well defined for all
  $w \in H^1(\mathbb R^N) \cap L^\infty(\mathbb R^N)$, and we have 
  \begin{align}
    \label{eq:dEdtw}
    {d \widetilde E[w^\eps(\cdot, t)] \over dt} = - \int_{\mathbb R^N}
    |w^\eps_t|^2 dx.
  \end{align}
  
  We claim that
  \begin{align}
    \label{eq:EEt}
    E[v_1 + w^\eps(\cdot, t)] = E[v_1] + \widetilde E[w^\eps(\cdot,
    t)],
  \end{align}
  for each $t \geq 0$, and thus by \eqref{eq:dEdtw} equation
  \eqref{eq:dEdt} holds for $\underline u_1^\eps$. Caution is needed
  here, since with our general assumptions on $f$ we have very little
  information about the decay of the ground states as
  $|x| \to \infty$. In particular, it is not a priori clear if
  $E[v_1 + w]$ is well defined for all
  $w \in H^1(\mathbb R^N) \cap L^\infty(\mathbb R^N)$. Hence we need
  to use some minimal regularity (at infinity) possessed by the ground
  states in the sense of Definition \ref{d:gs}.

  As was established in the proof of Lemma \ref{l:unst}, we have
  $\nabla v \in H^2(\mathbb R^N; \mathbb R^N)$ for every ground state
  $v$. Hence, by \eqref{stationary} we also have
  $f(v) \in L^2(\mathbb R^N)$.  Then, using Taylor formula it is easy
  to see that the integral of $V(v_1 + w)$ makes sense. Thus, we can
  write
  \begin{align}
    \label{eq:tEE}
    E[v_1 + w] - E[v_1] = \int_{\mathbb R^N} \left( \frac12 |\nabla
    w|^2 + \nabla v_1 \cdot \nabla w + V(v_1 + w) - V(v_1) \right) dx.
  \end{align}
  Again, using the fact that
  $\nabla v_1 \in H^2(\mathbb R^N; \mathbb R^N)$, we can integrate the
  second term in the right-hand side of \eqref{eq:tEE} by parts and
  use \eqref{stationary} to arrive at \eqref{eq:EEt}.

  Now, since $\underline v^\eps_1$ is a subsolution of \eqref{main}
  and by construction is a strict subsolution in $B_R(0)$, the
  function $\underline u_1^\eps(x, t)$ is strictly monotonically
  increasing in $t$ for each $x \in \mathbb R^N$. In particular,
  $v_1 < \underline u_1^\eps(\cdot, t)$ for each $t > 0$.  Since
  $\underline u_1^\eps(x, t)$ is bounded above for each
  $x \in \mathbb R^N$, by standard parabolic regularity
  $\underline u_1^\eps(\cdot, t)$ converges to a solution $v_3$ of
  \eqref{stationary} strongly in $C^1(\mathbb R^N)$ as $t \to \infty$.
  Again, by comparison principle $v_1 < v_3 < v_2$.

  We now show that $v_3$ is also a ground state. Indeed, by the
  decrease of energy we have for any $R > 0$
  \begin{align}
    \label{eq:E0BR}
    +\infty > E[\underline u_1^\eps(\cdot, 0)] \geq \int_{B_R(0)} \left(
    \frac12 |\nabla \underline u_1^\eps(x, t)|^2 + V(\underline
    u_1^\eps(x, t)) \right) dx + \int_{\mathbb R^N \backslash B_R(0)} V(\underline
    u_1^\eps(x, t)) \, dx. 
  \end{align}
  In view of the fact that $\underline u_1^\eps(x, t) < v_2(x)$ for
  all $x \in \mathbb R^N$, for every $R_0 > 0$ large enough we have
  $V(\underline u_1^\eps(x, t)) < \alpha$ for all $t \geq 0$ and all
  $|x| > R_0$. Recall that by our assumptions the function $V(u)$ is
  monotone for all $u \in (0, \alpha)$. Therefore, the last term in
  \eqref{eq:E0BR} can be bounded from below as follows:
  \begin{align}
    \label{eq:Vbelow}
    \int_{\mathbb R^N \backslash B_R(0)} V(\underline u_1^\eps(x, t)) \,
    dx
    \geq \min \left\{ 0, \  \int_{\mathbb R^N \backslash B_{R_0}(0)}
    V(v_2) \, dx \right\}, \qquad 
    \forall R > R_0.
  \end{align}
  Then, passing to the limit as $t \to \infty$ in \eqref{eq:E0BR}, we
  obtain
  \begin{align}
    \label{eq:ERVtlim}
    \int_{B_R(0)} \left(
    \frac12 |\nabla v_3|^2 + V(v_3) \right) dx \leq E[\underline
    u_1^\eps(\cdot, 0)] + \int_{\mathbb R^N \backslash B_{R_0}(0)}
    |V(v_2)| \, dx, 
  \end{align}
  for all $R > R_0$.  Furthermore, since
  $|V(v_3(x))| \leq |V(v_2(x))|$ for all $|x| > R_0$, passing to the
  limit as $R \to \infty$ in \eqref{eq:ERVtlim} and using Lebesgue
  monotone convergence theorem in the first term and Lebesgue
  dominated convergence theorem in the second term, we get
  $|\nabla v_3| \in L^2(\mathbb R^N)$ and
  $V(v_3) \in L^1(\mathbb R^N)$, so that $v_3$ is also a ground state.


  Finally, by Lemma \ref{l:unst} there exists $\delta > 0$
  sufficiently small and a supersolution $\overline v_3^\delta(x, t)$
  such that $\underline v_1^\eps < \overline v_3^\delta < v_3$.
  Therefore, by comparison principle we have
  $\underline u_1^\eps(x, t) < \overline v_3^\delta(x)$ for every
  $x \in \mathbb R^N$ and $t > 0$. But this contradicts the fact that
  $\underline u_1^\eps(\cdot, t) \to v_3$ uniformly as $t \to \infty$.
\end{proof}

We also establish strict positivity of the energy of ground states,
using a kind of Hamiltonian identity for \eqref{stationary} (see a
related discussion in \cite{cabre05}).

\begin{lemma}
  \label{l:vpos}
  Let $f \in C^1[0, \infty)$ and let $v$ be a ground state in the
  sense of Definition \ref{d:gs}. Then
  \begin{align}
    \label{eq:Evz}
    E[v] = {1 \over N} \int_{\mathbb R^N} |\nabla v|^2 \, dx > 0.
  \end{align}
\end{lemma}

\begin{proof}
  Let $v_i := \partial v / \partial x_i$, $i = 1, \ldots, N$, and for
  $R > 0$ let $\chi_R \in C^\infty_c(\mathbb R^N)$ be a cutoff
  function such that $0 \leq \chi_R \leq 1$,
  $\| \nabla \chi_R \|_{L^\infty(\mathbb R^N)} \leq C$ for some
  $C > 0$ independent of $R$, $\chi_R(x) = 1$ for all $|x| < R$ and
  $\chi_R(x) = 0$ for all $|x| > R + 1$. We multiply
  \eqref{stationary} by $\chi_R v_i$ and integrate over all
  $x \in \mathbb R^N$ such that $x_i < \xi$, for a fixed
  $\xi \in \mathbb R$. With the help of the fact that
  $\nabla v \in H^2(\mathbb R^N; \mathbb R^N)$ demonstrated in the
  proof of Lemma \ref{l:unst}, this yields, after a number of
  integrations by parts,
  \begin{align}
    \label{eq:vxi}
    0 & = \int_{ \{ x_i = \xi \} } \chi_R \left( v_i^2 - V(v) \right)
        d \mathcal H^{N-1}(x)  \notag \\
      & \qquad  - \int_{ \{ x_i < \xi \} } \chi_R \nabla v \cdot \nabla v_i
        \, dx - \int_{ \{ x_i < \xi \} } v_i \nabla \chi_R \cdot
        \nabla v \, dx +  \int_{ \{ x_i < \xi \} } V(v) {\partial
        \chi_R \over \partial x_i} dx \notag \\
      & =  \int_{ \{ x_i = \xi \} } \chi_R \left( v_i^2 - \frac12
        |\nabla v|^2 - V(v) \right)
        d \mathcal H^{N-1}(x)  \notag \\
      & \qquad  - \int_{ \{ x_i < \xi \} } v_i \nabla \chi_R \cdot
        \nabla v \, dx +  \int_{ \{ x_i < \xi \} } \left( \frac12
        |\nabla v|^2 + V(v) \right)  {\partial
        \chi_R \over \partial x_i} dx.
  \end{align}
  Since $\nabla \chi_R$ is uniformly bounded and supported on $\mathbb
  R^N \backslash B_R(0)$, by integrability of $|\nabla v|^2$ and
  $V(v)$ the last line in \eqref{eq:vxi} goes to zero when $R \to
  \infty$. Therefore, by Fubini's theorem and Lebesgue dominated
  convergence theorem we have for a.e. $\xi \in \mathbb R$
  \begin{align}
    \label{eq:Exi}
    \int_{ \{ x_i = \xi \} } v_i^2 
    d \mathcal H^{N-1}(x)  = \int_{ \{ x_i = \xi \} } \left( \frac12
    |\nabla v|^2 + V(v) \right)
    d \mathcal H^{N-1}(x).
  \end{align}
  Finally, integrating \eqref{eq:Exi} over all $\xi \in \mathbb R$, we
  get, again, by Fubini's theorem,
  \begin{align}
    \label{eq:Exi2}
    E[v] = \int_{\mathbb R^N} v_i^2 
    \, dx.  
  \end{align}
  In view of the fact that this identity holds for each
  $i = 1, \ldots, N$, summing up over all $i$ yields the statement.
\end{proof}

\begin{rmk}
  We note that by the argument in the proof of Lemma \ref{l:vpos}, for
  every ground state $v$ the function
  $\varphi(n) := \| n \cdot \nabla v \|_{L^2(\mathbb R^N)}$ is
  independent of $n$, for every $n \in \mathbb S^{N-1}$. This is
  consistent with radial symmetry of solutions of \eqref{stationary}
  known for many specific choices of $f$.
\end{rmk}

We will need the following simple non-existence result.

\begin{lemma}
  \label{l:gsnon}
  Let $f \in C^1[0, \infty)$ and suppose that there exist
  $0 \leq \alpha < \beta$ such that $f(u) \geq 0$ for all
  $u \in (\alpha, \beta)$. Then \eqref{stationary} has no non-constant
  radial symmetric-decreasing solutions with range in
  $(\alpha, \beta)$ whenever $N \leq 2$.
\end{lemma}

\begin{proof}
  The proof is elementary via the ordinary differential equation in
  \eqref{eq:vmu}. Let $v_\mu$ be a solution of \eqref{eq:vmu}
  satisfying $\alpha < v_\mu < \beta$. If $N = 1$, then $v_\mu(r)$ is
  concave for all $r > 0$. Since it is also strictly decreasing, we
  will necessarily have $v_\mu(r_0) = \alpha$ for some $r_0 > 0$,
  contradicting our assumption that $v_\mu(r)$ solves the equation for
  all $r > 0$ with $\alpha < v_\mu(r) < \beta$.

  If, on the other hand, $N = 2$, then with $s = \ln r$ as a new
  variable the solution of \eqref{eq:vmu} obeys (with a slight abuse
  of notation, we still denote the solution as $v_\mu(s)$)
  \begin{align}
    \label{eq:vmus}
    v_\mu''(s)+ e^{2 s} f(v_\mu(s)) = 0, \qquad v_\mu'(s) \leq 0, \quad
    \alpha < v_\mu(s) < \beta, \quad -\infty < s < +\infty.
  \end{align}
  Once again, $v_\mu(s)$ is concave and strictly decreasing, which is
  a contradiction.
\end{proof}

To conclude this section, we state the Poincar\'e type inequality
characterizing the exponentially weighted Sobolev spaces, which is a
straightforward generalization of \cite[Lemma 2.2]{mn:cms08} to the
whole space.

\begin{lemma}
  \label{l:poinc}
  Let $c > 0$ and let $u \in H^1_c(\mathbb R^N)$. Then for every open
  set $\Omega \subseteq \mathbb R^{N-1}$ there holds
  \begin{align}
    \label{eq:poinc}
    \int_R^\infty \int_\Omega e^{cz} u_z^2 \, dy \, dz 
    \geq {c^2 \over 4} 
    \int_R^\infty \int_\Omega  e^{cz} u^2 \, dy \, dz,
  \end{align}
  for every $R \in [-\infty, +\infty)$.
\end{lemma}

\section{Propagation: Proof of Theorems \ref{t:radfront}, \ref{t:igni}
  and \ref{t:omls}}
\label{s:prop}

We begin with the proof of Theorem \ref{t:igni}, which uses an
adaptation of the arguments from \cite{M2004,mn:cms08,mz:nodea13} to
the problem in $\mathbb R^N$. The key notion used in the proof is that
of a wave-like solution. 

\begin{defn}
  \label{d:wl}
  We call the solution $u(x, t)$ of \eqref{main} and \eqref{initial}
  {\em wave-like}, if $u(\cdot, T) \in H^1_c(\mathbb R^N)$ and
  $\Phi_c[u(\cdot,T)]<0$ for some $c>0$ and $T \geq 0$.
\end{defn}

\noindent We want to show that for radial symmetric-decreasing solutions of
\eqref{main} the wave-like property implies propagation whenever
$f'(0) \leq 0$.

Our first lemma connects the wave-like property of solutions with the
sign of their energy.

\begin{lemma}
  \label{l:wl}
  Let $u(x, t)$ be the solution of \eqref{main} satisfying
  \eqref{initial}. Suppose that $\phi \in L^2_{c_0}(\mathbb R^N)$ for
  some $c_0 > 0$ and suppose that there exists $T \geq 0$ such that
  $E[u(\cdot,T)]<0$. Then $u(x, t)$ is wave-like.
\end{lemma}
\begin{proof}
  This lemma is a multidimensional extension of \cite[Lemma
  3.2]{mz:nodea13} to the general nonlinearities in \eqref{bistable}.
  By Proposition \ref{p:exist} we have
  $u(\cdot, T) \in H^1(\mathbb R^N) \cap H^1_{c_0}(\mathbb R^N) \cap
  L^\infty(\mathbb R^N)$,
  so, in view of \eqref{bistable}, for every $\eps > 0$ there exists
  $R > 0$ such that
  \begin{align}
    \label{Repso1}
    & \int_{ \{ |x| > R \cap \{z > 0 \}\}  } e^{c_0z} \left( \frac12 |\nabla
      u(x, T)|^2 + V^+(u(x, T)) \right) dx < {\eps \over 4}, \\
    & \int_{ \{ |x| > R \cap \{z < 0 \}\}  } \left( \frac12 |\nabla
      u(x, T)|^2 + V^+(u(x, T)) \right) dx < {\eps \over 4}, \\
  \end{align}
  where $V^+(u) := \max \{ V(u), 0 \}$. Hence
  \begin{align}
    \label{Repso2}
    \int_{ \{ |x| > R\}  } e^{cz} \left( \frac12 |\nabla
    u(x, T)|^2 + V^+(u(x, T)) \right) dx < {\eps \over 2}, 
  \end{align}
  for every $c \in (0, c_0)$. 

  Possibly increasing the value of $R$, we also have
  \begin{align}
    \int_{ \{ |x| < R\} } \left( \frac12 |\nabla
    u(x, T)|^2 + V(u(x, T)) \right) dx < -\eps, 
  \end{align}
  provided that $\eps$ is small enough. Therefore, it is possible to
  choose $c \in (0, c_0)$ sufficiently small such that 
  \begin{align}
    \int_{ \{ |x| < R\} } e^{cz} \left( \frac12 |\nabla
    u(x, T)|^2 + V(u(x, T)) \right) dx < -{\eps \over 2}. 
  \end{align}
  Combining this with \eqref{Repso2} yields $\Phi_c[u(\cdot, T)] < 0$,
  proving the claim.
\end{proof}

\begin{rmk}
  \label{r:wl}
  If $u(x, t)$ is a solution of \eqref{main} satisfying
  \eqref{initial} that is wave-like, then we also have
  $\Phi_c[u(\cdot, t)] < 0$ for all $t \geq T$.
\end{rmk}

\begin{proof}
  Since $u(\cdot, T) \in H^1_c(\mathbb R^N)$ and
  $\Phi_c[u(\cdot, T)] < 0$, by Proposition \ref{p:exist} we have that
  \eqref{eq:dPhicdt} holds. Hence, if $\tilde u$ is defined in
  \eqref{eq:utilde}, we have
  $\Phi_c[u(\cdot, t)] = e^{c^2 t} \Phi_c[\tilde u(\cdot, t)] \leq
  e^{c^2 t} \Phi_c[\tilde u(\cdot, T)] = e^{c^2 (t - T)}
  \Phi_c[u(\cdot, T)] < 0$.
\end{proof}

We next show that the level sets of radial symmetric-decreasing
wave-like solutions propagate with positive speed. For such solutions,
the leading and the trailing edges defined in \eqref{eq:Rdp} and
\eqref{eq:Rdm} coincide:
$R_\delta^+(t) = R_\delta^-(t) =: R_\delta(t)$. We use the convention
that $R_\delta(t) := 0$ if $u(x, t) < \delta$ for all
$x \in \mathbb R^N$.

\begin{lemma}
  \label{spread}
  Let $f'(0) \leq 0$ and assume that \eqref{negpotential} holds. Let
  $u(x, t)$ be the solution of \eqref{main} satisfying \eqref{initial}
  and (SD), and suppose $u(\cdot, T) \in H^1_c(\mathbb R^N)$ and
  $\Phi_c[u(\cdot,T)]<0$ for some $c>0$ and $T \geq 0$.  Then for
  every $\delta \in (0, 1)$ and every $c' \in (0, c)$ there is
  $R_0 \in \mathbb R$ such that
  \begin{align}
    \label{eq:Rd}
    R_\delta(t) > c' t + R_0,
  \end{align}
  for all $t \geq 0$.
\end{lemma}

\begin{proof}
  Generalizing the definition in 
  \eqref{eq:thstar}, let
  \begin{align}
    \label{eq:thetastar2}
    \theta^* := \inf \left\{ u > 0 \ : \ V(u) < 0
    \right\}, 
  \end{align}
  and observe that by our assumptions $\theta^* \in [\theta_0,
  1)$. Next, define
  \begin{align}
    \label{eq:thc}
    \theta_c := \inf \left\{ u > 0 \ : \ V(u) + {c^2 u^2 \over 8} < 0
    \right\}. 
  \end{align}
  We claim that $\theta_c \in (\theta^*, 1)$. In particular, we have
  $\theta_c > \theta_0$. Indeed, clearly $\theta_c > \theta^*$ if
  $\theta^* > 0$. At the same time, since $f'(0) \leq 0$, we have
  $\theta_c > 0$.  Furthermore, by Lemma \ref{l:poinc} and Remark
  \ref{r:wl} there holds
  \begin{align}
    \label{phicc}
    0 > \Phi_c[u(\cdot, t)] \geq \int_{\mathbb R^N} e^{cz} \left( {c^2
    u^2(x, t) \over 8} + V(u(x, t)) \right) dx,
  \end{align}
  for all $t \geq T$. Therefore, passing to the limit $t \to \infty$
  in \eqref{phicc} and using Corollary \ref{c:limsup}, in view of
  \eqref{bistable} we conclude that $\theta_c < 1$.

  Now, by \eqref{eq:dPhicdt} we have for any $t > T$
  \begin{eqnarray}
    e^{-c (R_{\theta_c} (t) - ct)} \Phi_{c} [\tilde u(\cdot, T)]
    \geq e^{-c (R_{\theta_c}(t) - ct)} \Phi_{c} [\tilde u(\cdot,
    t)]  = e^{-c R_{\theta_c}(t)} \Phi_c[u(\cdot,t)],
  \end{eqnarray}
  where $\tilde u$ is defined in \eqref{eq:utilde}. Again, using Lemma
  \ref{l:poinc}, (SD) and noting that
  $\min_{u \geq 0} V(u) = V(1) < 0$ by \eqref{negpotential}, we obtain
  \begin{align}
    \Phi_c[u(\cdot, t)] 
    & \geq  \int_{ \{ |y| < R_{\theta_c(t)} \} \cap \{ z <
      R_{\theta_c(t)} \} } 
      e^{cz}V(u(x, t))\, dx \notag \\ 
    & \quad + \int_{ \{ |y| < R_{\theta_c(t)} \} \cap \{ z >
      R_{\theta_c(t)} \} } 
      e^{cz}\left( {c^2 u^2(x, t) \over 8} + V(u(x, t))
      \right) \, dx \notag \\
    & \quad + \int_{ \{ |y| > R_{\theta_c(t)} \} } 
      e^{cz}\left( {c^2 u^2(x, t) \over 8} + V(u(x, t))
      \right) \, dx\nonumber\\
    &\geq
      {C_N V(1) \over c} R_{\theta_c}^{N-1} (t) e^{c R_{\theta_c}(t)},
  \end{align}
  for some $C_N > 0$ depending only on $N$. Thus, for every $t>T$ we
  have
  \begin{equation}
    e^{-c (R_{\theta_c}(t) - ct)} \Phi_{c} [\tilde u(\cdot, T)]
    \geq  {C_N V(1) \over c} R^{N-1}_{\theta_c}(t).
  \end{equation}
  Dividing this inequality by a negative quantity
  $\Phi_{c} [\tilde u(\cdot, T)]$ and taking the logarithm of both
  sides, we obtain
  \begin{equation}\label{moveleading}
    R_{\theta_c}(t) + \frac {N-1}{c} \ln R_{\theta_c}(t) \geq
    ct + \frac{1}{c} \ln \frac{c \Phi_c [\tilde u(\cdot, T)]}{C_N V(1)}.
  \end{equation}
  As $t \rightarrow \infty$, the right-hand side of
  (\ref{moveleading}) goes to positive infinity, which implies that
  $\displaystyle \lim_{t\rightarrow\infty} R_{\theta_c}(t) =\infty$.
  Then $R_{\theta_c}(t)$ dominates in the left-hand side of
  (\ref{moveleading}) and, therefore, for any $c' \in (0, c)$ we have
  $R_{\theta_c}(t) > c' t$ for any sufficiently large $t$. This proves
  the desired result for all $\delta \in (0, \theta_c]$.

  To complete the proof, we need to show that \eqref{eq:Rd} also holds
  for all $\delta \in (\theta_c, 1)$. We note that by (SD) we have
  $\hat u(x, t) > \theta_c$ for all $x \in B_R(0)$ and $t \geq T'$,
  with some $T' > 0$ sufficiently large depending on $R > 0$, where
  $\hat u(y, z, t) := u(y, z + c't, t)$. At the same time, since
  $\theta_c > \theta_0$, we have that
  $\underline{u} : B_R(0) \times [T', \infty) \to [0, 1)$ solving
  \begin{align}
    & \underline{u}_t = \Delta \underline{u} + c' \underline{u}_z +
      f(\underline{u}), \quad (x, t) \in  B_R(0) \times
      [T', \infty), \label{uund} \\
    & \underline{u}(x, T') = \theta_c, \quad x \in B_R(0), \\
    & \underline{u}(x, t) = \theta_c, \quad   (x, t) \in  \partial
      B_R(0) \times  [T', \infty),  \label{uundbc} 
  \end{align}
  is a monotonically increasing in $t$ subsolution for $\hat u(x, t)$
  in $B_R(0) \times [T', \infty)$. In particular, by standard
  parabolic regularity we have
  $\underline{u}(x, t) \to \underline{u}_R^\infty(x)$ from below as
  $t \to \infty$ uniformly in $x \in B_R(0)$, where
  $\underline{u}_R^\infty(x) < 1$ is a stationary solution of
  \eqref{uund} and \eqref{uundbc}. By standard elliptic regularity
  \cite{gilbarg}, the latter, in turn, constitute an increasing family
  of solutions of \eqref{uund} that converge locally uniformly to a
  limit solution $\underline{u}^\infty(x) > \theta_c$ in all of
  $\mathbb R^N$ as $R \to \infty$. At the same time, from the fact
  that $f(u) > 0$ for all $u \in [\theta_c, 1)$ we conclude that
  $\underline{u}^\infty(x) = 1$ (use the solution of $u_t = f(u)$ with
  $u(x, 0) = \theta_c$ as a subsolution).  Hence, by comparison
  principle, we have $\lim_{t \to \infty} u(y, z + c't, t) = 1$,
  yielding the claim in view of (SD).
\end{proof}

\begin{cor}
  \label{c:wto1}
  Let $f'(0) \leq 0$ and assume that \eqref{negpotential} holds. Let
  $u(x, t)$ be a wave-like solution satisfying (SD). Then
  $\displaystyle \lim_{t\rightarrow\infty}u(\cdot, t)=1$ locally
  uniformly in $\mathbb{R}^N$.
\end{cor}

Finally, using a truncation argument similar to the one in \cite[Lemma
3.4]{mz:nodea13}, we can construct a wave-like subsolution for a
solution of \eqref{main} whose energy becomes negative at some
$t$. Then, applying Corollary \ref{c:wto1} to that subsolution and
using comparison principle, we arrive at the following result.

\begin{lemma}\label{test}
  Let $f'(0) \leq 0$ and assume that \eqref{negpotential} holds. Let
  $u(x, t)$ be the solution of \eqref{main} satisfying \eqref{initial}
  and (SD), and suppose that there exists $T \geq 0$ such that
  $E[u(\cdot,T)]<0$. Then
  $\displaystyle \lim_{t\rightarrow\infty}u(\cdot, t)=1$ locally
  uniformly in $\mathbb{R}^N$.
\end{lemma}

Lemma \ref{test} essentially constitutes the statement of part (ii) of
Theorem \ref{t:igni}. The proof of part (i) then comes from the
following lemma that estimates the energy dissipation rate for radial
symmetric-decreasing solutions that propagate.

\begin{lemma}
  \label{unboundedenergy}
  Let $f'(0) \leq 0$ and assume that \eqref{negpotential} holds. Let
  $u(x, t)$ be the solution of \eqref{main} satisfying \eqref{initial}
  and (SD), and suppose that $u(\cdot, t_n) \to 1$ locally uniformly
  in $\mathbb{R}^N$ for some sequence of $t_n \to \infty$. Then
  $\displaystyle\lim_{t\rightarrow\infty}E[u(\cdot,t)]= -\infty$.
\end{lemma}
\begin{proof}
  We argue by contradiction. Suppose that
  $\displaystyle \lim_{t\rightarrow\infty}u(\cdot, t)=1$ locally
  uniformly in $\mathbb{R}^N$ and $E[u(\cdot,t)]$ is bounded below.
  Fix $\eps \in (0, 1 - \theta^*)$, where $\theta^*$ is defined in
  \eqref{eq:thetastar2}, and $R > 0$, and consider
  \begin{align}
    \label{eq:udel}
    \phi_{\eps,R}(x) = 
    \begin{cases}
      1 - \eps, & |x| < R, \\
      (1 - \eps) (R + 1 - |x|), & R \leq |x| \leq R + 1, \\
      0, & |x| > R + 1.
    \end{cases}
  \end{align}
  It is easy to see that there exists $R = R_\eps$ such that
  $E[\phi_{\eps,R_\eps}] < 0$.  Now let $u_\eps(x, t)$ be the solution
  of \eqref{main} and \eqref{initial} with
  $\phi = \phi_{\eps, R_\eps}$. By Lemma \ref{l:wl}, $u_\eps(x, t)$ is
  wave-like.  Therefore, by Lemma \ref{spread} we have
  $R_{\theta_c}^\eps(t) > c' t + R_0$ for all $c' \in (0, c)$ and all
  $t \geq 0$, with some $R_0 \in \mathbb R$ independent of $t$, where
  $\theta_c$ is defined in \eqref{eq:thc} and $R_\delta^\eps(t)$ is
  the leading edge of $u_\eps(x, t)$.

  Since $u(\cdot, t_n) \to 1$ as $t \to \infty$ locally uniformly in
  $\mathbb R^N$, there exists $T_\eps \geq 0$ such that
  $u(x, T_\eps) \geq u_\eps(x, 0)$ for all $x \in \mathbb R^N$.
  Therefore, $u_\eps(x, t - T_\eps)$ is a subsolution for $u(x, t)$
  for all $t \geq T_\eps$ and, consequently, by comparison principle
  we have $R_{\theta_c}(t) > c t / 2$ for all $t \geq t_0$, for some
  $t_0 \geq T_\eps$.  This implies that $u(x,t)\geq\theta_c$ for all
  $t \geq t_0$ and $x \in \mathbb R^N$ such that $|x| \leq c t /2$.

  By \eqref{eq:dEdt}, for any $\alpha > 0$ there exists
  $t_\alpha \geq 0$ such that
  \begin{equation}
    \int^{\infty}_{t_\alpha}\int_{\mathbb{R}^N}u^2_t(x,t) \, dx
    \, dt \leq \alpha^2.
  \end{equation}
  Let us take
  \begin{align}
    \label{eq:al}
    \alpha = 2^{-(3N+4)/2} c^{N / 2} \theta_c \left| B_1(0) \right|^{1/2} .  
  \end{align}
  We also take $t_1 \geq 0$ sufficiently large such that
  $t_1 \geq \max\{t_\alpha,t_0\}$ and $r_0=R_{\theta_c / 2}(t_1)>1$.
  In addition, we take $T$ sufficiently large such that
  $T>\max\{ t_0, 1 \}$ and $r_0 < cT/4$. Finally, we take $t_2=t_1+T$
  and $r=cT/2$.  Since $t_2>t_1 \geq t_\alpha$, by Cauchy-Schwarz
  inequality we have
  \begin{eqnarray}\label{eq:utcontr}
    \int^{t_2}_{t_1}\int_{ \{ |x|<r \} }| u_t(x,t)|\, dx \, dt
    &\leq& \sqrt{r^N |B_1(0)| (t_2-t_1)} \left( \int^{t_2}_{t_1}
           \int_{|x|<r}u^2_t(x,t) \, dx \, dt \right)^{1/2}\nonumber\\
    &\leq& \sqrt{r^N |B_1(0)| (t_2-t_1)} \left( \int^{\infty}_{t_\alpha}
           \int_{\mathbb{R}^N}u^2_t(x,t) \, dx \, dt \right)^{1/2}\nonumber\\
    &\leq& \alpha \sqrt{r^N |B_1(0)| (t_2-t_1)}\nonumber\\
    &=&\frac{\theta_c |B_1(0)|}{4} \left( \frac{c}{4} \right)^N T^{(N+1)/2}.
  \end{eqnarray}

  On the other hand, we also have
  \begin{eqnarray}
    \int^{t_2}_{t_1}\int_{ \{ |x|<r \} } | u_t(x,t)| \, dx \, dt 
    & \geq & \int^{t_2}_{t_1} \int_{ \{ \frac{cT}{4} < |x| <
             \frac{cT}{2} \} }  | u_t(x,t)| \, dx 
             \, dt  \nonumber\\
    & \geq & \int_{ \{ \frac{cT}{4} < |x| <
             \frac{cT}{2} \} } 
             (u(x,t_2) - u(x,t_1)) \, dx.
  \end{eqnarray}
  Since $t_2>T>t_0$, we have $u(x,t_2)\geq \theta_c$ for
  $|x| \leq cT/2$, and by the definition of $r_0$ and $T$ we have
  $u(x,t_1)<\theta_c/2$ for $cT/4 < |x| < cT/2$.  So we have
  \begin{equation}
    \int^{t_2}_{t_1}\int_{ \{ |x|<r \} } |u_t(x, t)| \, dx \, dt \geq
    \frac{\theta_c |B_1(0)|}{2} \left( \frac{c}{4} \right)^N T^N,
  \end{equation}
  which contradicts \eqref{eq:utcontr}, because $T>1$ and $N \geq 1$.
\end{proof}

\begin{proof}[Proof of Theorem \ref{t:igni}] 
  We just need to verify that the assumptions of Lemma \ref{test} and
  Lemma \ref{unboundedenergy} are satisfied. If $\theta_0 = 0$, then
  by \eqref{bistable} we must have $f(u) > 0$ for all $u \in (0,1)$,
  and so \eqref{negpotential} is clearly satisfied. On the other hand,
  if $\theta_0 > 0$, then by \eqref{bistable} we must have
  $f'(0) \leq 0$.
\end{proof}

\begin{proof}[Proof of Theorem \ref{t:omls}]
  By Theorem \ref{t:igni}, either
  $\lim_{t \to \infty} E[u(\cdot, t)] = -\infty$ or
  $\lim_{t \to \infty} E[u(\cdot, t)] \geq 0$. Indeed, if
  $\lim_{t \to \infty} E[u(\cdot, t)] \in (-\infty, 0)$, then by
  Theorem \ref{t:igni}(ii) we have $u(\cdot, t) \to 1$ locally
  uniformly in $\mathbb R^N$. However, this contradicts Theorem
  \ref{t:igni}(i), since in this case one would have
  $\lim_{t \to \infty} E[u(\cdot, t)] = -\infty$. 

  If $\lim_{t \to \infty} E[u(\cdot, t)] = -\infty$, then there exists
  $T > 0$ such that $E[u(\cdot, T)] < 0$. Hence $u(\cdot, t) \to 1$
  locally uniformly in $\mathbb R^N$ by Theorem \ref{t:igni}(ii). This
  establishes the first alternative. In the second alternative, we
  have $E[u(\cdot, t)] \geq 0$ for all $t \geq 0$.  Therefore, by
  \eqref{eq:dEdt} there exists a sequence of $t_n \in [n, n+1)$ such
  that $u_t (\cdot, t_n) \to 0$ in $L^2(\mathbb R^N)$. In turn, by
  standard parabolic regularity one can extract a subsequence
  $t_{n_k}$ from this sequence such that $u(\cdot, t_{n_k}) \to v$ in
  $C^1_{loc}(\mathbb R^N)$ as $k \to \infty$. Following the usual
  argument for gradient flows, from \eqref{eq:mainweak} we then
  conclude that $v$ solves \eqref{stationary} distributionally and,
  hence, classically \cite{gilbarg}. Furthermore, $v$ is radial
  symmetric-decreasing, and, taking into account Corollary
  \ref{c:limsup}, we have $v(x) = v_\mu(|x|)$ for some
  $\mu \in [0, 1]$ and all $x \in \mathbb R^N$, where $v_\mu$ solves
  \eqref{eq:vmu}. Note that all radial symmetric-decreasing solutions
  of \eqref{stationary} can be parametrized by $\mu = v(0)$. In
  particular, by continuous dependence of the solutions of
  \eqref{eq:vmu} on $\mu$ in the $C^1_{loc}(\mathbb R)$ topology the
  set of all $\mu$'s for which the solution of \eqref{eq:vmu} exists
  is closed.

  By Theorem \ref{t:igni}(i), $\mu = 1$ is impossible when
  $E[u(\cdot, t)] \geq 0$ for all $t \geq 0$. Hence the set of all
  $\mu$'s corresponding to the limits of $u(\cdot, t_{n_k})$ is
  contained in $[0, 1)$. Denoting by $\omega(\phi) \subseteq [0, 1)$
  the set of all limits of $u(\cdot, t_{n_k})$ in
  $C^1_{loc}(\mathbb R^N)$ parametrized by $\mu = v(0)$, which
  coincides with the $\omega$-limit set of $u(x, t)$ (cf. Proposition
  \ref{holder} and standard parabolic regularity), by the usual
  properties of $\omega$-limit sets we have that
  $\omega(\phi) = [a, b]$ for some $0 \leq a \leq b < 1$, i.e., that
  $\omega(\phi) \subseteq [0, 1)$ is closed and connected. Thus
  \begin{align}
    \label{eq:liminfsup}
    \lim_{n \to \infty}    \inf_{\mu \in \omega(\phi)} \sup_{x \in
    B_R(0)} \big| u(x, t_n) - v_\mu(|x|) \big| \to 0, 
  \end{align}
  for any $R > 0$. In view of Proposition \ref{holder}, this completes
  the proof of the statement in the second alternative.
\end{proof}

\begin{proof}[Proof of Theorem \ref{t:radfront}]
  Since for $N = 1$ the result was established in \cite[Proposition
  2.4]{mz:nodea13}, in the rest of the proof we assume that
  $N \geq 2$. By rotational symmetry, the upper bound on
  $R_\delta^+(t)$ follows exactly as in \cite[Proposition
  5.2]{mn:cms08}, noting that
  $u(\cdot, t) \in H^2(\mathbb R^N) \cap H^2_{c_0}(\mathbb R^N)$ for
  each $t > 0$.  To prove the matching lower bound, for each
  $c \in (0, c^\dag)$ we construct a test function
  $u_c \in H^1_c(\mathbb R^N)$ which is radial symmetric-decreasing
  and satisfies $\Phi_c[u_c] < 0$. Then the result follows by Lemma
  \ref{spread}.

  Let $\bar u_c(z)$ be a one-dimensional minimizer from
  \cite[Proposition 2.3]{mz:nodea13}, which, e.g., by simple phase
  plane arguments is non-increasing in $z$. For $R > 0$, we define
  \begin{align}
    \label{eq:uc}
    u_c^R(x) := \bar u_c(|x| - R), \qquad x \in \mathbb R^N.
  \end{align}
  In particular, $\text{supp} \, (u_c^R) = \overline B_R(0)$, and
  $u_c^R$ satisfies (SD). We also note that by the definition of
  $\bar u_c$ and boundedness of $\bar u_c$ and $\bar u_c'$ there
  exists $K > 0$ such that for all $R > 0$ sufficiently large we have
  \begin{align}
    \label{uczneg}
    m_{c,R} := \int_0^\infty e^{cz} \left( \frac12 |\bar u_c'(z -
    R)|^2 + V(\bar u_c(z - R)) \right) dz < -K e^{cR}.
  \end{align}
  Let us now evaluate $\Phi_c[u_c]$. Passing to the spherical
  coordinates, we obtain
  \begin{align}
    \label{eq:Phicuc1}
    \Phi_c[u_c] 
    & = \int_{\mathbb R^N} e^{cz} \left( \frac12 |\nabla
      u_c|^2 + V(u_c) \right) dx \notag \\
    & = |\mathbb S^{N-2}| \int_0^\infty \int_0^\pi e^{c r \cos \theta}
      \left( \frac12 | \bar u'_c(r - R) |^2 + V( \bar u_c(r - R) )
      \right) r^{N-1} \sin^{N-2} \theta \, d \theta \, dr \notag \\
    & = |\mathbb S^{N-2}| R^{N-1} m_{c,R} \int_0^\pi e^{-c R (1 - \cos
      \theta)} \sin^{N-2} \theta
      \, d \theta +  |\mathbb S^{N-2}| \int_0^\infty \int_0^\pi e^{c r
      \cos \theta} \notag \\
    & \qquad \times 
      \left( \frac12 | \bar u'_c(r - R) |^2 + V( \bar u_c(r - R) )
      \right) (r^{N-1} - R^{N-1}) \sin^{N-2} \theta \, d \theta \, dr
      \notag \\ 
    & + |\mathbb S^{N-2}| R^{N-1} \int_0^\infty \int_0^\pi e^{c
      r } \left( e^{-c r (1 - \cos \theta )} -  e^{-c R (1 - \cos
      \theta )}  \right) \notag \\ 
    & \qquad \times 
      \left( \frac12 | \bar u'_c(r - R) |^2 + V( \bar u_c(r - R) )
      \right)  \sin^{N-2} \theta \, d \theta \, dr.
  \end{align}
  We proceed to estimate, using the fact that
  $e^{-c r (1 - \cos \theta )} \simeq e^{-c r \theta^2 / 2}$ for
  $\theta \ll 1$ (the details are left to the reader):
  \begin{align}
    e^{-c R} R^{-{N-1 \over 2}} \Phi_c[u_c]  
    & \leq -C K + C' R^{-1}, 
  \end{align}
  for some $C, C' > 0$ independent of $R$. Therefore, choosing $R$
  sufficiently large yields the claim.
\end{proof}

\section{Bistable nonlinearities: Proof of Theorem \ref{t:thrbist}} 
\label{s:bistab}

We now proceed to the sharp threshold results. For bistable
nonlinearities satisfying \eqref{negpotential}, we establish existence
of a sharp threshold in Theorem \ref{t:thrbist} under (TD). We define
\begin{align}
  \label{eq:Sigma0}
  \Sigma_0 & := \{\lambda \in [0, \lambda^+]
             :\;u_{\lambda}(\cdot,t)\rightarrow0
             \; \text{as} \;t\rightarrow\infty\;\text{uniformly
             in}\; \mathbb{R}^N\}, \\ 
  \label{eq:Sigma1}
  \Sigma_1 & := \{\lambda \in [0, \lambda^+]
             :\;u_{\lambda}(\cdot,t)\rightarrow1 
             \; \text{as} \;t\rightarrow\infty\;\text{locally
             uniformly in}\; \mathbb{R}^N\}, \\
  \label{eq:Sigma*}
  \Sigma_* & := [0, \lambda^+] \backslash (\Sigma_0 \cup \Sigma_1).  
\end{align}
Our goal is to prove that the threshold set $\Sigma_*$ is a single
point and to characterize the threshold solution.

\begin{proof}[Proof of Theorem \ref{t:thrbist}]
  By Theorem \ref{t:omls}, for every $\lambda \in [0, \lambda^+]$ we
  have either $E[u_\lambda (\cdot, T)] < 0$ for some $T > 0$, or
  $E[u_\lambda (\cdot, t)] \geq 0$ for all $t > 0$. In the first case,
  we have $u_\lambda (\cdot, t) \to 1$ locally uniformly as
  $t \to \infty$ and, therefore, $\lambda \in \Sigma_1$. In the second
  case, we have $u_\lambda (\cdot, t) \not\to 1$ locally uniformly, so
  $\lambda \not\in \Sigma_1$.

  Consider the case of $\lambda \in \Sigma_1$. Note that by (P3), the
  set $\Sigma_1$ is non-empty. As was mentioned in the preceding
  paragraph, we have $E[u_\lambda (\cdot, T)] < 0$ for some $T > 0$.
  Then by continuous dependence of the solution on the initial data we
  also have $E[u_{\lambda'} (\cdot, T)] < 0$ for all
  $\lambda' \in (0, \lambda^+]$ in a sufficiently small neighborhood
  of $\lambda$. Furthermore, by (P2) and comparison principle, for
  every $0 < \lambda_1 < \lambda_2 < \lambda^+$ we have
  $u_{\lambda_1} (\cdot, t) < u_{\lambda_2} (\cdot, t) $ for all
  $t > 0$. Therefore, if $\lambda_1 \in \Sigma_1$, then so is
  $\lambda_2$. This means that there exists
  $\lambda_*^+ \in (0, \lambda^+)$ such that
  $\Sigma_1 = (\lambda_*^+, \lambda^+]$.

  At the same time, we know that $\lambda \in \Sigma_0$ if and only if
  there exists $T > 0$ such that $u_\lambda(0,T) < \theta_0$ (use the
  solution of $u_t = f(u)$ with $u(x, 0) = u_0 \in (0, \theta_0)$ as a
  supersolution). Again, by continuous dependence of solutions on the
  initial data, if $u_\lambda(0,T) < \theta_0$ for some
  $\lambda \in [0, \lambda^+)$, then $u_{\lambda'}(0,T) < \theta_0$ as
  well for all $\lambda' \in [0, \lambda^+)$ sufficiently close to
  $\lambda$, and by comparison principle we have
  $u_{\lambda_1} (\cdot, t) \to 0$ uniformly for all
  $0 < \lambda_1 < \lambda_2 < \lambda^+$, whenever
  $u_{\lambda_2} (\cdot, t) \to 0$, as $t \to \infty$. Thus, again, by
  (P3) there exists $\lambda_*^- \in (0, \lambda^+)$ such that
  $\Sigma_0 = [0, \lambda_*^-)$.

  Now, let
  $\lambda \in \Sigma_* = [\lambda_*^-, \lambda_*^+] \not=
  \varnothing$.
  We claim that under our assumptions $u_\lambda (\cdot, t) \to v$
  uniformly as $t \to \infty$, where $v$ is a ground state. Indeed,
  for a sequence of $t_n \to \infty$ as in the proof of Theorem
  \ref{t:omls}, let $u(\cdot, t_{n_k}) \to v$ in
  $C^1_{loc}(\mathbb R^N)$ for some $n_k \to \infty$, with
  $v(x) = v_\mu(|x|)$ and $v_\mu$ solving \eqref{eq:vmu} with
  $\mu = v(0)$. Note that since $\lambda \not\in \Sigma_0$, by
  comparison principle we have $u_\lambda(0, t_{n_k}) > \theta_0$ for
  all $k \in \mathbb N$ and, hence, $\mu \geq \theta_0$.  Furthermore,
  for every $k$ there is a unique $R_k > 0$ such that
  $u_\lambda(x, t_{n_k}) = \theta_0/2$ for $|x| = R_k$.  Also, there
  exists $R_0 \geq 0$ such that $\theta_0 < v_\mu(R_0) < \theta^*$,
  where $\theta^*$ is defined via \eqref{eq:thstar}, with the
  convention that $R_0 = 0$ if this inequality has no solution. Then,
  by (SD) there exists $k_0 \in \mathbb N$ such that
  $u_\lambda(x, t_{n_k}) < \theta^*$ for all $|x| > R_0$ and all
  $k \geq k_0$. In turn, this means that
  $V(u_\lambda(x, t_{n_k})) \geq 0$ and $R_k > R_0$ for all
  $k \geq k_0$ and $|x| > R_0$.

  By \eqref{eq:dEdt}, for any $T \in (0, 1)$ and $k \geq k_0$ we have
  for any $R \geq R_0$
  \begin{multline}
    \label{eq:Elam*bi}
    +\infty > E[u_\lambda (\cdot, T)] \geq E[u_\lambda (\cdot,
    t_{n_k})] \geq \int_{B_R(0)} \left( \frac12 | \nabla u_\lambda (x,
      t_{n_k})|^2 +
      V(u_\lambda (x, t_{n_k})) \right) dx \\
    + \int_{\mathbb R^N \backslash B_R(0)} V( u_\lambda (x, t_{n_k}))
    dx. \qquad
  \end{multline}
  On the other hand, since $V( u_\lambda (x, t_{n_k})) \geq V_0 > 0$
  for all $R_0 < |x| < R_k $ and $k \geq k_0$, where
  $V_0 := \min \{ V(\theta_0/2), V(v_\mu(R_0))\}$, from
  \eqref{eq:Elam*bi} with $R = R_0$ we get for all $k \geq k_0$:
  \begin{align}
    \label{eq:Elam**bi}
    +\infty > E[u_\lambda (\cdot, T)] \geq E[u_\lambda (\cdot,
    t_{n_k})] \geq \int_{B_{R_0}(0)} V(u_\lambda (x, t_{n_k})) \, dx +
    C_N (R_k^N - R_0^N) V_0,
  \end{align}
  for some $C_N > 0$ depending only on $N$. Thus, if
  $v \geq \theta_0$, we would have $R_k \to \infty$ as $k \to \infty$,
  contradicting \eqref{eq:Elam**bi}.

  We just demonstrated that $v(x) \to 0$ as $|x| \to \infty$. Then,
  passing to the limit in \eqref{eq:Elam*bi} as $k \to \infty$, we get
  \begin{align}
    \label{eq:Elam***bi}
    +\infty >  E[u_\lambda (\cdot, T)] \geq
    \int_{B_R(0)} \left( \frac12 | \nabla v|^2 + V(v) \right) dx, 
  \end{align}
  for all $R > R_0$ and $t > 0$. Then, sending $R \to \infty$, by
  Lebesgue dominated convergence theorem and monotonicity of
  $E[u_\lambda (\cdot, t)]$ we obtain that
  $ E[u (\cdot, t)] \geq E[v]$.  On the other hand, since
  $V(v(x)) > 0$ for all $|x| > R_0$, this implies that
  $V(v) \in L^1(\mathbb R^N)$ and $|\nabla v| \in L^2(\mathbb
  R^N)$. Thus, $v$ is a ground state.

  By the arguments above, every limit of $u_\lambda(\cdot, t_{n_k})$
  is a ground state independently of the choice of $n_k$. Therefore,
  from Theorem \ref{t:omls} we get $I \subseteq \Upsilon$. On the
  other hand, by (TD) this means that $a = b \in (0, 1)$. In
  particular, the limit $v$ is independent of $n_k$. Thus,
  $u_\lambda(\cdot, t_n) \to v$ as $n \to \infty$, and by Proposition
  \ref{holder} we have $u_\lambda(\cdot, t) \to v$ uniformly as
  $t \to \infty$. By standard parabolic regularity, this convergence
  is also in $C^1(\mathbb R^N)$.

  Finally, we claim that $\Sigma_*$ consists of only a single point,
  i.e., that $\lambda_*^- = \lambda_*^+ =: \lambda_*$. We argue by
  contradiction. Suppose, to the contrary, that
  $\lambda_*^- < \lambda_*^+$. Since $E[u_\lambda (\cdot, t)] \geq 0$
  for all $t > 0$ and all $\lambda \in \Sigma_*$, there exists a
  sequence $n_k \in \mathbb N$, $n_k \to \infty$ as $k \to \infty$,
  and two sequences, $t^-_{n_k} \in [n_k, n_k + 1)$ and
  $t^+_{n_k} \in [n_k, n_k + 1)$, such that
  $u_{\lambda_*^-} (\cdot, t^-_{n_k}) \to v^-$ (resp.
  $u_{\lambda_*^+} (\cdot, t^+_{n_k}) \to v^+$) in
  $C^1_{loc}(\mathbb R^N)$ as $k \to \infty$, where
  $v^-(x) = v_{\mu^-}(|x|)$ (resp. $v^+(x) = v_{\mu^+}(|x|)$) and
  $v_{\mu^-}$ (resp. $v_{\mu^+}$) solve \eqref{eq:vmu} for some
  $\mu^- \in [\theta_0, 1)$ (resp.  $\mu^+ \in [\theta_0, 1)$; cf. the
  arguments in the proof of Theorem \ref{t:omls}). Furthermore, by
  comparison principle we have $\mu^- \leq \mu^+$. Then, by
  Proposition \ref{holder} and standard parabolic regularity we also
  have $u_{\lambda_*^-} (\cdot, t_{n_k}) \to v^-$ (resp.
  $u_{\lambda_*^+} (\cdot, t_{n_k}) \to v^+$) in $C^1_{loc}(\mathbb R^N)$
  for $t_{n_k} = n_k + 2$, as $k \to \infty$.  Therefore, since
  $v^- \leq v^+$ and both $v^-$ and $v^+$ are ground states, by
  Corollary \ref{c:nomonot} we have $v^- = v^+ =: v_*$.

  Let us show that this gives rise to a contradiction. Let
  $w(x, t) := u_{\lambda_*^+} (x, t) - u_{\lambda_*^-} (x, t) > 0$ for
  all $x \in \mathbb R^N$ and $t > 0$. By \eqref{main}, $w(x, t)$
  solves
  \begin{align}
    \label{eq:w*}
    w_t = \Delta w + f'(\tilde u) w,
  \end{align}
  for some
  $u_{\lambda_*^-} (x, t) < \tilde u(x, t) < u_{\lambda_*^+} (x,
  t)$.
  On the other hand, since both $u_{\lambda_*^-} (x, t)$ and
  $u_{\lambda_*^+} (x, t)$ converge uniformly to $v_*$, we can use the
  arguments leading to \eqref{wunderR} in the proof of Lemma
  \ref{l:unst} to show that $w(x, t) > \eps \phi_0^R(x)$, where
  $\phi_0^R$ is as in \eqref{phiR0}, for all $x \in \mathbb R^N$ and
  $t > 0$ sufficiently large, contradicting the fact that
  $\| w (\cdot, t) \|_{L^\infty(\mathbb R^N)} \to 0$ as
  $t \to \infty$.

  It remains to establish the asymptotic behavior of the energy as
  $t \to \infty$. If $\lambda \in \Sigma_1$, then the statement is
  contained in Theorem \ref{t:omls}. On the other hand, if
  $\lambda \in \Sigma_0$, then we know that
  $u_\lambda(\cdot, t) \to 0$ in $L^2(\mathbb R^N)$ (use the solution
  of the heat equation as a supersolution for large $t$). Hence by
  \eqref{eq:L2H2} we also have $u_\lambda(\cdot, t + 1) \to 0$ in
  $H^1(\mathbb R^N)$, which implies that
  $E[u_\lambda(\cdot, t)] \to 0$ as $t \to \infty$ in this
  case. Finally, if $\lambda \in \Sigma_*$, then
  $u_\lambda(\cdot, t) \to v_*$ for some ground state $v_*$ and
  $E[u_\lambda(\cdot, t)] \geq E[v_*] > 0$ by \eqref{eq:Elam***bi} and
  Lemma \ref{l:vpos}.
\end{proof}

\section{Ignition nonlinearities: Proof of Theorems \ref{t:thrign2}
  and \ref{t:thrign3}}
\label{s:igni}

\begin{proof}[Proof of Theorem \ref{t:thrign2}]
  As in the proof of Theorem \ref{t:thrbist}, we can define the sets
  $\Sigma_0$, $\Sigma_1$ and $\Sigma_*$, and by the same argument we
  have $\Sigma_1 = (\lambda_*^+, \lambda^+]$ for some
  $\lambda_*^+\in (0, \lambda^+)$. Similarly, every solution in
  $\Sigma_0$ satisfies the linear heat equation for all sufficiently
  large $t$ and, therefore, we have $\Sigma_0 = [0, \lambda_*^+)$ for
  some $\lambda_*^-\in (0, \lambda^+)$.  Thus
  $\Sigma_* = [\lambda_*^-, \lambda_*^+] \not= \varnothing$. 

  Let now $\lambda \in \Sigma_*$, and notice that since
  $\lambda \not\in \Sigma_0$, by comparison principle we have
  $u_\lambda(0, t_{n_k}) > \theta_0$ for all $k$.  Then by the same
  arguments as in the proof of Theorem \ref{t:thrbist}, there exists a
  sequence of $t_n \in [n, n+1)$ and a sequence of $n_k \to \infty$
  such that $u_\lambda(\cdot, t_{n_k}) \to v$ as $k \to \infty$ for
  some $v(x) = v_\mu(|x|)$, where $v_\mu$ solves \eqref{eq:vmu} with
  $\mu = v(0) \geq \theta_0$. By Lemma \ref{l:gsnon}, if $N \leq 2$
  then $\mu = \theta_0$. Hence, in view of the uniqueness of the limit
  $v$ independently of $n_k$, we have
  $u_\lambda(\cdot, t) \to \theta_0$ locally uniformly as
  $t \to \infty$. The bound on the energy is contained in Theorem
  \ref{t:omls}.

  The proof is completed by showing that $\lambda_*^- = \lambda_*^+$,
  which can be done as in the proof of \cite[Theorem
  9]{mz:nodea13}. The latter relies on \cite[Lemma 4]{zlatos06}, which
  is valid in $\mathbb R^N$ for all $N \geq 1$.
\end{proof}

\begin{proof}[Proof of Theorem \ref{t:thrign3}]
  The proof proceeds as that of Theorem \ref{t:thrign2} up to the
  point when $u_\lambda(\cdot, t_{n_k}) \to v$ for
  $\lambda \in \Sigma_*$. However, in contrast to lower dimensions,
  for $N \geq 3$ there exist many solutions of \eqref{eq:vmu},
  including a continuous family of non-constant solutions with
  $v^\infty := v_\mu(\infty) \in [0, \theta_0)$
  \cite{BL1983,berestycki81}. In particular, there is at least one
  ground state \cite{BL1983}.

  We claim that $v$ is a ground state. Indeed, a priori we have
  $v^\infty \in [0, \theta_0]$. Consider first the case
  $v^\infty \in (0, \theta_0)$. Since
  $u_\lambda(\cdot, t_{n_k}) \not\to 0$, we have
  $u_\lambda(0, t_{n_k}) > \theta_0$ for all $k$. Therefore, there
  exists a unique $R_k > 0$ such that
  $u_\lambda(x, t_{n_k}) = v^\infty$ for $|x| = R_k$. Also, there
  exists a unique $R_0 > 0$ which solves
  $v^\infty < v_\mu(R_0) < \theta_0$.  Since
  $u_\lambda (\cdot, t_{n_k}) \to v$ in $C^1(B_{R_0}(0))$, by (SD) we
  also have $u_\lambda (\cdot, t_{n_k}) < \theta_0$ for all
  $|x| > R_0$ and all $k \geq k_0$, with some $k_0 \in \mathbb N$
  large enough, and, hence, for all $k \geq k_0$ we have $R_k > R_0$
  and $V(u_\lambda (x, t_{n_k})) = 0$ for all $|x| > R_0$.

  By \eqref{eq:dEdt}, for any $T \in (0, 1)$, $R > R_0$ and
  $k \geq k_0$ we have
    \begin{multline}
    \label{eq:Elam*igni}
    +\infty > E[u_\lambda (\cdot, T)] \geq E[u_\lambda (\cdot,
    t_{n_k})] = \int_{B_R(0)} \left( \frac12 | \nabla u_\lambda (x,
      t_{n_k})|^2 +
      V(u_\lambda (x, t_{n_k})) \right) dx \\
    + \frac12 \int_{\mathbb R^N \backslash B_R(0)} | \nabla u_\lambda
    (x, t_{n_k})|^2 dx . \qquad
  \end{multline}
  On the other hand, by \cite[Radial Lemma A.III]{BL1983} we can
  estimate the right-hand side of \eqref{eq:Elam*igni} from below as
  \begin{align}
    \label{eq:Elam**igni}
    +\infty >  E[u_\lambda (\cdot, T)] \geq \int_{B_{R_0}(0)}
    V(u_\lambda (x, t_{n_k})) \, dx + C_N  R_k^{N-2} |v^\infty|^2,
  \end{align}
  for all $k \geq k_0$, where $C_N > 0$ depends only on
  $N$. Therefore, if $v^\infty \in (0, \theta_0)$, we would have
  $R_k \to \infty$ as $k \to \infty$, contradicting
  \eqref{eq:Elam**igni}.

  Thus, we have either $v^\infty = 0$ or $v^\infty = \theta_0$. Let us
  consider the first case.  Passing to the limit in
  \eqref{eq:Elam*igni} as $k \to \infty$, we get
  \begin{align}
    \label{eq:Elam***igni}
    +\infty > E[u_\lambda (\cdot, T)] \geq
    \int_{B_R(0)} \left( \frac12 | \nabla v|^2 + V(v) \right) dx, 
  \end{align}
  for all $R \geq R_0$. Then, sending $R \to \infty$, by Lebesgue
  dominated convergence theorem we obtain that
  $ E[u (\cdot, T)] \geq E[v]$.  On the other hand, since
  $V(v(x)) = 0$ for all $|x| > R_0$, this implies that
  $V(v) \in L^1(\mathbb R^N)$ and $|\nabla v| \in L^2(\mathbb R^N)$.
  Thus, $v$ is a ground state. Furthermore, we claim that if $v'$ is a
  limit of $u_\lambda (\cdot, t_{n'_k})$ for another choice $t_{n'_k}$
  of a subsequence of $t_n$, then $v' = v$. Indeed, since
  $v_\mu(R_1) = \theta_0$ for some $R_1 > 0$, by continuous dependence
  of solutions of \eqref{eq:vmu} on $\mu$ we have
  $v_{\mu'}(\infty) = 0$ as well for all $\mu'$ in some small
  neighborhood of $\mu$, whenever $v_{\mu'}$ exists. Therefore, by
  (TD) the set $\Upsilon \cup \{v'(0) \}$ is disconnected,
  contradicting part 2 of the statement of Theorem \ref{t:omls},
  unless $v' = v$.

  Consider now the case $v^\infty = \theta_0$. Define $R_k > 0$ to be
  such that $u(x, t_{n_k}) = \theta_0 / 2$ for $|x| = R_k$, and
  observe that $R_k \to \infty$ as $k \to \infty$. Arguing as in the
  preceding paragraph, we have
    \begin{multline}
    \label{eq:Elam*igni2}
    +\infty > E[u_\lambda (\cdot, T)] \geq E[u_\lambda (\cdot,
    t_{n_k})] \geq \int_{B_{R_k}(0)} \left( \frac12 | \nabla u_\lambda
      (x, t_{n_k})|^2 +
      V(u_\lambda (x, t_{n_k})) \right) dx \\
    + C_N R_k^{N-2} |\theta_0|^2. \qquad
  \end{multline}
  Therefore, for every $M > 0$ there exists $k_0 \in \mathbb N$ and
  $R_0 := R_{k_0}$ such that $R_k \geq R_0$ for all $k \geq k_0$ and
   \begin{align}
    \label{eq:Elam**igni2}
     E_0[u_\lambda(x, t_{n_{k_0}})] := \int_{B_{R_0}(0)} \left( \frac12 |
     \nabla u_\lambda (x, t_{n_{k_0}})|^2 + 
     V(u_\lambda (x, t_{n_{k_0}})) \right) dx  \leq -M.
  \end{align}
  We now take $\underline u_\lambda(x, t)$ to be the solution of
  \eqref{main} on $B_{R_0}(0)$ for $t > t_{n_{k_0}}$ with
  $\underline u_\lambda(x, t_{n_{k_0}}) = u_\lambda(x, t_{n_{k_0}})$
  for all $x \in B_{R_0}(0)$ and
  $\underline u_\lambda(x, t) = \theta_0/2$ for all
  $x \in \partial B_{R_0}(0)$ and $t > t_{n_{k_0}}$. Possibly
  increasing the value of $k_0$, we also have that
  $u_\lambda(x, t) > \theta_0/2$ for all $x \in \partial B_{R_0}(0)$
  and $t > t_{n_{k_0}}$. Indeed, if not, then there is a sequence of
  $t_n' \to \infty$ such that $u_\lambda(x, t'_n) \leq \theta_0$ for
  all $|x| \geq R_0$ as $n \to \infty$. However, by the preceding
  arguments this would imply that $u_\lambda(\cdot, t'_n)$ converges
  to a ground state in $C^1_{loc}(\mathbb R^N)$ as $n \to \infty$,
  which contradicts our assumption of $v^\infty = \theta_0$. Thus,
  $\underline u_\lambda(x, t)$ is a subsolution for $u_\lambda(x, t)$
  for all $x \in B_{R_0}(0)$ and $t > t_{n_{k_0}}$, and by comparison
  principle we have $\underline u_\lambda(x, t) < u_\lambda(x, t)$ in
  $B_{R_0}(0)$. In addition, from the gradient flow structure of
  \eqref{main} on $B_{R_0}(0)$ we have
  $E_0[\underline u_\lambda(\cdot, t_{n_k})] \leq E_0[\underline
  u_\lambda(\cdot, t_{n_{k_0}})] = E_0[u_\lambda(\cdot, t_{n_{k_0}})]$
  for all $k \geq k_0$.  Therefore, since
  $u_\lambda(\cdot, t_{n_k}) \to v$ in $C^1(B_{R_0}(0))$ as
  $k \to \infty$ and $V(u)$ is a non-increasing function of $u$, by
  \eqref{eq:Elam**igni2} we conclude that
  \begin{align}
    -M \geq \lim_{k \to \infty} E_0[\underline u_\lambda (\cdot,
    t_{n_k})] \geq \lim_{k \to \infty} \int_{B_{R_0}(0)} V(u_\lambda (x,
    t_{n_k})) \, dx \notag \\
    =  \int_{B_{R_0}(0)} V(v) \, dx 
    \geq - \| V(v) \|_{L^1(\mathbb
    R^N)}.
  \end{align}
  However, by (V) this is a contradiction when $M$ is sufficiently
  large. Thus, $v^\infty = \theta_0$ is impossible.

  We thus established that $v$ is a ground state and that $v$ is the
  full limit of $u_\lambda(\cdot, t)$ as $t \to \infty$ for any
  $\lambda_* \in \Sigma_*$. The remainder of the proof follows as in
  the proof of Theorem \ref{t:thrbist}.
\end{proof}

\section{Monostable nonlinearities: Proof of Theorems \ref{t:mono2},
  \ref{t:mono30} and \ref{t:mono31}}
\label{s:mono}

In view of the hair-trigger effect for $f'(0) > 0$ \cite{aronson78},
in which case the statements of all the theorems trivially holds true
with $\lambda_* = 0$, it is sufficient to assume $f'(0) = 0$ in all
the proofs.

\begin{proof}[Proof of Theorem  \ref{t:mono2} and Theorem \ref{t:mono30}]
  Once again, we define the sets $\Sigma_1$, $\Sigma_0$ and $\Sigma_*$
  and note that by the same argument as in the proofs of the preceding
  theorems we have $\Sigma_1 = (\lambda_*^+, \lambda_+]$ for some
  $\lambda_*^+ \in (0, \lambda_+)$. At the same time, by Lemma
  \ref{l:gsnon} there are no positive solutions of \eqref{eq:vmu} for
  any $\mu \in (0,1)$ and $N \leq 2$. Similarly, by the assumption of
  Theorem \ref{t:mono30} there are no positive solutions of
  \eqref{eq:vmu} for any $\mu \in (0,1)$ and $N \geq 3$. Thus,
  $\Sigma_* = \varnothing$, and we have $\Sigma_0 = [0, \lambda_*^+]$.
\end{proof}

\begin{proof}[Proof of Theorem \ref{t:mono31}]
  The proof proceeds in the same fashion as before, establishing that
  $\Sigma_1 = (\lambda_*^+, \lambda_+]$ for some
  $\lambda_*^+ \in (0, \lambda_+)$. If $\lambda_*^+ = 0$, we are
  done. Otherwise, suppose that $\lambda_*^+ > 0$. If there exists
  $\lambda \in (0, \lambda_*^+]$ such that $u_\lambda(\cdot, t) \to 0$
  uniformly as $t \to \infty$, then by comparison principle
  $[0, \lambda] \subseteq \Sigma_0$. Let
  $\lambda_*^- \leq \lambda_*^+$ be the supremum of all such values of
  $\lambda$. Then either $\Sigma_0 = [0, \lambda_*^-)$ or
  $\Sigma_0 = [0, \lambda_*^-]$. In the second case and with
  $\lambda_*^- = \lambda_*^+$ we are done once again. Otherwise
  $\Sigma_* \not= \varnothing$.  Finally, by our assumptions and the
  arguments in the proofs of the preceding theorems, for every
  $\lambda \in \Sigma_*$ we have $u_\lambda(\cdot, t) \to v$ uniformly
  as $t \to \infty$, where $v$ is a ground state. Then by the
  arguments in the proof of Theorem \ref{t:thrbist} we have
  $\lambda_*^- = \lambda_*^+$.
\end{proof}

\section*{Acknowledgements}

This work was supported, in part, by NSF via grants DMS-0908279,
DMS-1119724 and DMS-1313687. CBM wishes to express his gratitude to
V. Moroz for many valuable discussions.

\bibliography{references2,../mura,../bio,../nonlin}
\bibliographystyle{plain}

\end{document}